\documentclass[a4paper,reqno]{amsart}

\usepackage{stmaryrd,amssymb,mathtools,booktabs,upref}
\usepackage[inline,shortlabels]{enumitem}
\usepackage[hidelinks]{hyperref}

\newtheorem{theorem}{Theorem}[section]
\newtheorem{proposition}[theorem]{Proposition}
\newtheorem{lemma}[theorem]{Lemma}
\newtheorem{corollary}[theorem]{Corollary}

\theoremstyle{definition}
\newtheorem{definition}[theorem]{Definition}

\theoremstyle{remark}
\newtheorem{remark}[theorem]{Remark}
\newtheorem*{notation}{Notation}

\makeatletter
\g@addto@macro\bfseries{\boldmath}
\@namedef{subjclassname@2020}{\textup{2020} Mathematics Subject Classification}
\makeatother

\numberwithin{equation}{section}

\newcommand{\eqbreak}[1][2]{\\&\hspace{#1em}}
\newcommand{\eqand}[1][1]{\hspace{#1em}\textnormal{and}\hspace{#1em}}
\newcommand{\eqcond}[2][2]{\hspace{#1em}\textnormal{#2}}

\newcommand{\any}{\,\cdot\,}
\newcommand{\hook}{\mathbin{\lrcorner}}

\newcommand{\bC}{\mathbb{C}}
\newcommand{\bR}{\mathbb{R}}
\newcommand{\bN}{\mathbb{N}}

\newcommand{\lie}[1]{\mathfrak{#1}}
\newcommand{\mfa}{\lie{a}}
\newcommand{\mfg}{\lie{g}}
\newcommand{\mfh}{\lie{h}}
\newcommand{\mfr}{\lie{r}}
\newcommand{\mfp}{\lie{p}}
\newcommand{\aff}{\lie{aff}}
\newcommand{\gl}{\lie{gl}}
\newcommand{\derg}{\mfg'}

\newcommand{\spa}[1]{\operatorname{span}(#1)}

\renewcommand{\Re}{\operatorname{Re}}
\renewcommand{\Im}{\operatorname{Im}}

\newcommand{\cA}{\mathcal{A}}
\newcommand{\cH}{\mathcal{H}}
\newcommand{\Vc}{V \otimes \bC}

\DeclareMathOperator{\End}{End}
\DeclareMathOperator{\Hom}{Hom}
\DeclareMathOperator{\ad}{ad}
\DeclareMathOperator{\codim}{codim}
\DeclareMathOperator{\diag}{diag}
\DeclareMathOperator{\im}{im}

\newcommand{\id}{\mathrm{id}}
\newcommand{\inc}{\mathrm{inc}}

\DeclarePairedDelimiterX{\inp}[2]{\langle}{\rangle}{#1,#2}
\DeclarePairedDelimiter{\abs}{\lvert}{\rvert}
\DeclarePairedDelimiter{\norm}{\lVert}{\rVert}

\setlist{nosep,font=\upshape}

\begin{document}

\title{Two-step solvable SKT shears}

\author[M.~Freibert]{Marco Freibert}

\address[M.~Freibert]{Mathematisches Seminar\\
Christian-Albrechts-Universit\"at zu Kiel\\
Ludewig-Meyn-Strasse 4\\
D-24098 Kiel\\
Germany}

\email{freibert@math.uni-kiel.de}

\author[A.~F. Swann]{Andrew Swann}

\address[A.~F. Swann]{Department of Mathematics, and DIGIT\\
Aarhus University\\
Ny Munkegade 118, Bldg 1530\\
DK-8000 Aarhus C\\
Denmark}

\email{swann@math.au.dk}

\date{}

\begin{abstract}
  We use the shear construction to construct and classify a wide range
  of two-step solvable Lie groups admitting a left-invariant SKT
  structure.
  We reduce this to a specification of SKT shear data on Abelian Lie
  algebras, and which then is studied more deeply in different cases.
  We obtain classifications and structure results for \( \mfg \)
  almost Abelian, for derived algebra \( \derg \) of
  codimension~\( 2 \) and not \( J \)-invariant, for \( \derg \)
  totally real, and for \( \derg \) of dimension at most~\( 2 \).
  This leads to a large part of the full classification for
  two-step solvable SKT algebras of dimension six.
\end{abstract}

\subjclass[2020]{Primary: 53C55, Secondary: 22E25}

\maketitle

\tableofcontents

\section{Introduction}

Recent years have seen a rising interest in non-K\"ahler Hermitian
manifolds \( (M,g,J) \).
The most studied case is \emph{strong K\"ahler with torsion} (SKT)
geometry, which may be characterized by
\( \partial\overline{\partial} \omega=0 \), where \( \omega \) is the
associated fundamental two-form.
An equivalent description is obtained if one uses the Bismut
connection \( \nabla^B \) on \( (M,g,J) \), which is the unique
connection with \( \nabla^B g=0 \), \( \nabla^B J=0 \) and the
associated torsion tensor~\( T^B \) totally skew-symmetric (after
lowering one index).
Then \( T^B= -J^*d\omega \) and \( (M,g,J) \) is SKT if and only if
\( d T^B=0 \).

The latter property is used in physics in the context of
supersymmetric theories, cf.~for example \cite{GHR},
\cite{HP},~\cite{St}.
These structures also appear in generalized K\"ahler geometry
\cite{AG}, \cite{C}, \cite{FP},~\cite{GHR}, and in two complex
dimensions they are ``standard'' in the sense of Gauduchon~\cite{Gau},
i.e.\ any conformal class of any Hermitian metric on a compact complex
surface admits an SKT metric.
Moreover, any compact Lie group possesses a left-invariant SKT
structure, see e.g.~\cite{MaSw}.

However, in higher dimensions, the existence of an SKT metric is not
automatically guaranteed anymore.
For example, looking at the invariant case on nilpotent Lie groups,
only \( 4 \) out of the \( 34 \) six-dimensional nilpotent Lie
algebras admit (non-K\"ahler) SKT structures~\cite{FPS}.
Also the eight-dimensional nilpotent case has been classified
in~\cite{EFV}, and in both cases all nilpotent SKT Lie algebras are
two-step nilpotent, leading to the question whether this is a general
feature in arbitrary dimensions~\cite{FV}.

A more general class of Lie algebras that so far only has partially
been studied, is the solvable one.
Here, a full classification has only been achieved in dimension four
by the second author in cooperation with T.~B.
Madsen~\cite{MaSw}.
This reflects the much more complicated classification of solvable Lie
algebras.
In fact, in dimension six, the classification of all solvable Lie
algebras is already very long and complicated \cite{Mu2},~\cite{Tu}
and in higher dimensions there is no classification at all, nor is
there hope that there will be a full list in the future.
Hence, so far, only classifications of solvable SKT Lie algebras under
severe extra conditions in dimensions six and higher are known.
For example, \cite{FOU}~classifies the six-dimensional solvable Lie
algebras admitting an SKT structure and a non-zero holomorphic
\( (3,0) \)-form, and in~\cite{FKV} it is shown that no complex
solvable Lie group may admit an invariant SKT structure.
Moreover, \cite{AL}~provides a classification of the almost Abelian
Lie algebras admitting an SKT structure, cf.\ also~\cite{FP} for a
more explicit list in the six-dimensional case.

In this article, we determine, for various natural classes of two-step
solvable (almost Hermitian) Lie algebras, the ones which are SKT\@.
This includes, in particular, also a new proof for the classification
in the almost Abelian case leading to a more compact description of
the almost Abelian Lie algebras admitting an SKT structure in
Corollary~\ref{co:almostAbelian}.
Moreover, our results eventually lead to a classification of all but
one class of six-dimensional two-step solvable Lie algebras admitting
an SKT structure in Theorem~\ref{th:6d2stepsolvSKT}.

To obtain the classification results, we apply the shear construction
introduced in our joint paper~\cite{FrSw}.
This construction is a generalisation of the twist construction
invented by the second-named author in~\cite{Sw}.
The twist looks at double fibrations of the form
\( M\leftarrow P \rightarrow W \) of two principal \( A \)-bundles
with commuting \( A \)-actions, \( A \)~being a connected Abelian Lie
group, and transfers \( A \)-invariant tensor fields from~\( M \)
to~\( W \) by requiring that their pullbacks to \( W \) agree on the
horizontal distribution in~\( P \).
The principal \( A \)-bundles and all other data may be obtained by
``twist data'' on \( M \) and the twist is well-adapted to nilpotent
Lie groups and nilmanifolds in the sense that any nilmanifold may be
constructed from a torus by several twists increasing the nilpotent
step length of the associated nilpotent Lie group by one with each
step.
Note that the second author used the twist construction in the above
mentioned paper~\cite{Sw} to construct new examples of compact
simply-connected SKT manifolds out of well-known ones.

The shear construction now replaces the group actions of the twist by
flat connections and the principal bundles by certain types of
submersions and is well-adapted to solvable Lie groups in the sense
that any simply-connected solvable Lie group may be obtained from
\( \bR^n \) by consecutive shears increasing the solvable step length
by one.
Again, the shear construction can be completely encoded by ``shear
data'' on \( M \) and so any two-step solvable SKT Lie algebra can be
obtained by one shear from the standard flat (left-invariant) K\"ahler
structure on~\( \bR^n \) and appropriate shear data on~\( \bR^n \).

Let us now give a more detailed overview of the results in this paper
and outline its organisation.

In Section~\ref{sec:2stepsolvable}, we recall the shear construction
on arbitrary Lie algebras and give a full description of the possible
shear data on Abelian Lie algebras.

Section~\ref{sec:data2stepsolvable} first derives the general
conditions for shear data on \( \bR^{2n} \) to shear the flat K\"ahler
structure to an SKT-structure.
For this, we decompose \( \mfg \) into
\( \mfg'_J\oplus \mfg'_r\oplus J\mfg'_r\oplus U_J \), where
\( \mfg'_J:=\mfg'\cap J\mfg' \), \( \mfg'_r \) is the orthogonal
complement of \( \mfg'_J \) in \( \mfg' \) and \( U_J \) is the
orthogonal complement of \( \mfg'_J\oplus \mfg'_r\oplus J\mfg'_r \) in
\( \mfg \).
Afterwards, we derive some general consequences on the different
components of the shear data with respect to this splitting.
In particular, we show that \( \ad(J\mfg'_r) \) preserves
\( \mfg'_J \) and is diagonalisable over the complex on this subspace.

Section~\ref{sec:UJ=0} deals with the case \( U_J=0 \), which is
equivalent to \( \mfg'\oplus J\mfg'=\mfg \).
We derive the general form of a two-step solvable SKT Lie algebra of
this form and then deal individually with the cases
\( \codim(\mfg')=1 \) and \( \codim(\mfg')=2 \).
The first case is the almost Abelian one and we give a full
classification of all SKT almost Abelian Lie algebras in
Theorem~\ref{th:AlmostAbelianSKTLAs} as well as an easy
characterisation of the almost Abelian Lie algebras admitting an SKT
structure in Corollary~\ref{co:almostAbelian}.
The second case may equivalently be described as \( \mfg' \) being of
codimension two and \( \mfg' \) not being \( J \)-invariant.
We classify all such two-step solvable SKT Lie algebras in
Theorem~\ref{th:codim2notJinv}.

In Section~\ref{sec:totallyreal}, we consider the case
\( \mfg'_J=\{0\} \), so \( \mfg' \) is totally real.
Note that then the complex structure of an SKT structure is Abelian
and classify such SKT Lie algebras in
Theorem~\ref{th:totallyrealgeneralresult} up to some minor remaining
equation.
In particular, we show that then the Lie algebra is of the form
\( r\aff_{\bR}\oplus \mfh \) for some nilpotent Lie algebra
\( \mfh \), \( r:=\dim([\mfg',J\mfg']) \).
If \( [\mfg',J\mfg'] \) is of codimension \( \ell \) at most two in
\( \mfg' \), we are also able to solve the remaining equation and
\( \mfg \) is isomorphic to \( r \aff_{\bR}\oplus \bR^{2n-2r} \)
(\( \ell=0 \)), \( r\aff_{\bR}\oplus \mfh_3\oplus \bR^{2n-2r-3} \)
(\( \ell=1 \)) or to
\( r\aff_{\bR}\oplus 2\mfh_3\oplus \bR^{2n-2r-6} \) (\( \ell=2 \)),
respectively.
We use our results to obtain a classification of all SKT Lie algebras
with one-dimensional commutator or with two-dimensional commutator
which is not \( J \)-invariant.
The corresponding Lie algebras are isomorphic to
\( \aff_{\bR}\oplus \bR^{2n-2} \), \( \mfh_3\oplus \bR^{2n-3} \) or to
\( 2\aff_{\bR}\oplus \bR^{2n-4} \),
\( \aff_{\bR}\oplus \mfh_3\oplus \bR^{2n-5} \),
\( 2\mfh_3\oplus \bR^{2n-6} \), respectively.

Section~\ref{sec:2dcomplexcommutator} determines all SKT Lie algebras
with \( \dim(\mfg')=2 \) and \( \mfg' \) being \( J \)-invariant in
Theorem~\ref{th:2dJinv}.
The associated Lie algebra is isomorphic either to
\( \mfr_{3,0}'\oplus \bR^{2n-3} \) or to \( \mfh\oplus \bR^{2n-2k} \)
with a six- or eight-dimensional (i.e.\ \( k\in \{3,4\} \)) nilpotent
SKT Lie algebra.
Hence, we obtain a full list of these Lie algebras using the results
in~\cite{FPS} and~\cite{EFV}.

Finally, in Section~\ref{sec:6d2stepsolvSKT} we put our previous
results together, solve one additional case in
Theorem~\ref{th:6dwith3dnottotallyrealcommutator} (for
\( \dim(\mfg)=6 \), \( \dim(\mfg')=3 \), \( \mfg'\neq \mfg'_r \)) and
so are able to give a full classification of six-dimensional two-step
solvable Lie algebras admitting an SKT structure in
Theorem~\ref{th:6d2stepsolvSKT} up to one particular case, where we
write down the resulting equations.

\section{Two-step solvable algebras}
\label{sec:2stepsolvable}

The general shear construction~\cite{FrSw}, is motivated by and
specialises to constructions of Lie groups \( H \) from left-invariant
data on a Lie group \( G \) of the same dimension.
The construction is set-up so that if \( G \) is solvable, then
\( H \) is also solvable with in general step length one greater than
that of~\( G \).
The construction may be understood purely at the algebraic level,
using left-invariant shear data, as follows.

Let \( \mfg \) be the Lie algebra of~\( G \).
Suppose \( \mfa_G \) and \( \mfa_P \) are Abelian Lie algebras of the
same dimension and we have \( (\xi,a,\omega,\gamma,\eta) \), where
\( \gamma\in \mfg^*\otimes \gl(\mfa_G) \),
\( \eta\in \mfg^*\otimes\gl(\mfa_P) \), \( \xi\colon \mfa_G\to G \) a
Lie algebra monomorphism, \( a\colon \mfa_G\to \mfa_P \) a Lie algebra
isomorphism and a two-form
\( \omega\in \Lambda^2 \mfg^*\otimes \mfa_P \) with values in
\( \mfa_P \).
Then \( (\xi,a,\omega,\gamma,\eta) \) defines left-invariant
\emph{shear data} on \( \mfg \) if and only \( \xi(\mfa_G) \) is an
ideal in \( \mfg \) and the following conditions are satisfied:
\begin{enumerate}[(Sa)]
\item\label{item:pullbackzero} \( \xi^*\omega = 0 \),
\item\label{item:omegadNclosed} \( d\omega+\eta\wedge \omega=0 \),
\item\label{item:etaandgamma}
  \( \eta=a \gamma a^{-1}-(\xi\hook \omega)a^{-1} \),
\item\label{item:gammazero} \( \gamma(\xi Y)=0 \) for all
  \( Y\in \mfa_G \),
\item\label{item:gammaandad} \( \xi\circ \gamma(X)=\ad(X)\circ \xi \)
  for all \( X\in \mfg \), and
\item\label{item:etaflat} \( d\eta+\eta\wedge\eta=0 \).
\end{enumerate}

The \emph{shear} \( H \) of \( G \) then has Lie algebra
\( \mfh = \mfp/\mathring\xi(\mfa_{G}) \), where \( \mfp \) is the
extension of \( \mfg \) by \( \mfa_{P} \) via the two-form
\( \omega \) and \( \mathring\xi\colon \mfa \to \mfp \) is a lift of
\( \xi \).
More precisely, as a vector space \( \mfp = \mfg \oplus \mfa_{P} \)
with \( \mfa_{P} \) an Abelian ideal,
\( [X,Y]_{\mfp} = [X,Y]_{\mfg} - \omega(X,Y) \),
\( [X,Z]_{\mfp} = \eta(X)Z \).
The lift \( \mathring\xi \) is given as \( \tilde\xi + \rho\circ a \),
where \( \rho\colon \mfa_{P} \to \mfp \) is the inclusion and
\( \tilde\xi\colon \mfa_{G} \to \mfp \) is the horizontal lift
of~\( \xi \) with respect to the flat connection on
\( E = G \times \mfa_{G} \) defined by~\( \gamma \).

The above construction and data simplifies considerably if we start
with \( G \) Abelian, as we will now explain.
So take \( \mfg = \bR^{N} \) for some~\( N \).
Then as \( \xi \) is injective, we may assume that
\( \mfa = \mfa_{G} \) is a Lie subalgebra of \( \mfg=\bR^N \), so an
arbitrary subspace, and that \( \xi=\inc \) is the inclusion.
Next, since \( a \) is invertible, we may take \( \mfa_P=\mfa \) and
\( a=\id_{\mfa} \), the identity map.
It turns out that given the form \( \omega \) the remaining pieces of
shear data \( \gamma \) and \( \eta \) are uniquely determined.
To explain this, let \( U \) be a vector space complement to
\( \mfa \) in \( \mfg \):
\begin{equation*}
  \mfg = \mfa + U.
\end{equation*}
Then
\begin{equation*}
  \Lambda^{2}\mfg^{*} = \Lambda^{2}\mfa^{*} \oplus U^{*} \wedge
  \mfa^{*} \oplus \Lambda^{2}U^{*},
\end{equation*}
and we correspondingly decompose
\begin{equation*}
  \omega = \omega_{-1} + \omega_{0} + \omega_{1},
\end{equation*}
with \( \omega_{-1} \in \Lambda^{2}\mfa^{*}\otimes \mfa \),
\( \omega_{0} \in (U^{*} \wedge \mfa^{*}) \otimes \mfa \subset U^{*}
\otimes \End{\mfa} \) and
\( \omega_{1} \in \Lambda^{2}U^{*} \otimes \mfa \).

\begin{lemma}
  \label{le:AbelianSKT}
  For \( \mfg = \bR^{N} \) Abelian,
  \( (\inc,\id_{\mfa},\omega = \omega_{-1} + \omega_{0} +
  \omega_{1},\gamma,\eta) \) defines left-invariant shear data if and
  only if \( \omega_{-1} = 0 \), \( \gamma = 0 \),
  \( \eta = -\omega_{0} \) and
  \begin{equation}
    \label{eq:Abeliansheardata}
    \cA(\omega(\omega(\any,\any),\any))=0,
  \end{equation}
  where \( \cA \) is the anti-symmetrisation map over all arguments.
\end{lemma}

\begin{proof}
  We need to check
  conditions~\ref{item:pullbackzero}--\ref{item:etaflat}.
  The first condition~\ref{item:pullbackzero} is exactly
  \( \omega_{-1} = 0 \).
  Now as \( \mfg \) is Abelian condition~\ref{item:gammaandad}, says
  \( \inc\circ\gamma(X)=\ad(X)\circ \inc = 0 \) for all
  \( X\in \mfg=\bR^N \), so \( \gamma=0 \).
  This implies condition~\ref{item:gammazero} is automatically
  satisfied and condition~\ref{item:etaandgamma} is equivalent to
  \( \eta=-\inc\hook \omega = - \omega_{0} \in U^{*}
  \otimes \End{\mfa} \).
  Considering condition~\ref{item:omegadNclosed}, we get
  \( 0=d\omega+\eta\wedge \omega=-\omega_0\wedge \omega \).
  But if we now regard \( \omega_0 \) as an element of
  \( (U^*\wedge \mfa^*)\otimes \mfa\subseteq \Lambda^2 \mfg^*\otimes
  \mfa \), then this condition reads
  \( \cA(\omega_0(\omega(\any,\any),\any))=0 \).
  As \( \omega_1(\omega(\any,\any),\any)=0 \), this is equivalent
  to~\eqref{eq:Abeliansheardata}.

  Finally, we need to consider condition~\ref{item:etaflat}, i.e.\
  \( 0=d\eta+\eta\wedge\eta=\eta\wedge \eta \).
  If we regard \( \omega_0 \) as an element of
  \( (U^*\wedge \mfa^*)\otimes \mfa\subseteq \Lambda^2 \mfg^*\otimes
  \mfa \), then this condition is equivalent to
  \( \cA(\omega_0(\omega_0(\any,\any),\any))=0 \), which is automatic
  under~\eqref{eq:Abeliansheardata}.
\end{proof}

As all two-step solvable algebras are obtained as shears of Abelian
algebras~\cite[Proposition~2.5]{FrSw}, we see that shear data for such
a two-step algebra consists of \( \mfg = \bR^{N} \), a subalgebra
\( \mfa \leq \mfg \) and
\( \omega = \omega_{0} + \omega_{1} \in (U^{*}\wedge\mfa^{*} +
\Lambda^{2}U^{*}) \otimes \mfa \)
satisfying~\eqref{eq:Abeliansheardata}.

\begin{definition}
  We call \( (\mfa,\omega) \) \emph{two-step shear data} on
  \( \mfg = \bR^{N} \), with \( \mfa \) the \emph{shearing algebra}
  and \( \omega \) the \emph{shearing form}.
  If
  \( \omega \in (U^{*}\wedge\mfa^{*} + \Lambda^{2}U^{*}) \otimes \mfa
  \) does not necessarily satisfy~\eqref{eq:Abeliansheardata}, then we
  call \( (\mfa,\omega) \) \emph{pre-shear data}.
\end{definition}

\section{Data for two-step solvable SKT algebras}
\label{sec:data2stepsolvable}

The shear construction becomes a useful computational tool in the
presence of geometric structures.
In general, provided they satisfy an invariance property with respect
to \( \mfa_{G} \), they may be lifted horizontally to \( \mfp \) and
then descend to the shear \( \mfh = \mfp/\mathring\xi\mfa_{G} \);
starting with a tensor \( \alpha \), we denote the resulting tensor
on~\( \mfh \) by~\( \alpha_{\mfh} \) and say that \( \alpha \) and
\( \alpha_{\mfh} \) are \emph{\( \cH \)-related}, where
\( \cH \)~refers the horizontal distribution of a certain connection
on~\( \mfp \).
In~\cite{FrSw}, we provided a number of simple formulae for computing
various derived properties of~\( \alpha_{\mfh} \) in terms of
properties of \( \alpha \) on~\( \mfg \) and the shear data.

Let \( (\mfa,\omega) \) be two-step shear data on \( \mfg = \bR^N \)
and assume that \( N=2n \) for some \( n\in\bN \).
We equip \( \mfg = \bR^{2n} \) with the standard flat K\"ahler
structure \( (g,J,\sigma=g(J\any,\any)) \).
We will determine the conditions for such a shear~\( \mfh \) to be SKT
(strong K\"ahler with torsion) if we start with the standard on
\( \bR^{2n} \).

First recall, that an SKT structure on an algebra \( \mfh \) consists
of a Hermitian structure \( (g_{\mfh},J_{\mfh}),\sigma_{\mfh} \), in
particular \( J_{\mfh} \) is integrable, such that the torsion
three-form \( c_{\mfh} = -J_{\mfh}^{*}d\sigma_{\mfh} \) is closed.
This then defines a left-invariant SKT structure on the group~\( H \),
and conversely each such left-invariant structure arises this way.
Note that any K\"ahler structure is SKT, and that any almost Hermitian
structure on the Abelian algebra \( \mfg = \bR^{2n} \) is K\"ahler.

\begin{lemma}
  \label{le:SKTshears}
  Let \( (\mfa,\omega) \) be two-step shear data on
  \( \mfg = \bR^{2n} \).
  Let \( (g,J,\sigma) \) be a flat K\"ahler structure on
  \( \mfg = \bR^{2n} \).
  Then the induced almost Hermitian geometry
  \( (g_{\mfh},J_{\mfh},\sigma_{\mfh}) \) on the shear \( \mfh \) is
  SKT if and only if the following two equations hold on~\( \mfg \):
  \begin{gather}
    J^*\omega =\omega-J\circ J.\omega, \label{eq:AbelianSKTshear1}\\
    \cA\Bigl(g\bigl(J^*\omega(\any,\any), \omega(\any,\any)\bigr) +
    2g\bigl(J^*\omega(\omega(\any,\any),\any),\any \bigr)\Bigr) = 0,
    \label{eq:AbelianSKTshear2}
  \end{gather}
  where \( (J.\omega)(X,Y) = - \omega(JX,Y) - \omega(X,JY) \), and
  \( \cA \) is anti-symmetrisation over all arguments.
  Furthermore, any two-step solvable SKT algebra~\( \mfh \) may be
  obtained as a shear of a flat K\"ahler structure
  on~\( \mfg = \bR^{2n} \) in this way.
\end{lemma}

\begin{proof}
  As \( \gamma = 0 \), the invariance conditions
  in~\cite[\S\S3.2--3.3]{FrSw} are just that the tensors are
  left-invariant on~\( G \).
  The condition for \( J_{\mfh} \) to be integrable is equivalent
  to~\eqref{eq:AbelianSKTshear1} by~\cite[Proposition~3.13]{FrSw},
  since \( J \)~is integrable.
  If \( \alpha \in \Lambda^{k}\mfg^{*} \), then as \( \mfg \) is
  Abelian, we have \( d\alpha = 0 \).
  It follows from~\cite[Corollary~3.11]{FrSw}, that
  \( d\alpha_{\mfh} \in \Lambda^{k+1}\mfh^{*} \) is \( \cH \)-related
  to the differential form
  \( d_{\mfh}\alpha := - (\xi \hook \alpha) \wedge \omega \in
  \Lambda^{k+1}\mfg^{*} \), where
  \( \xi = \inc \colon \mfa \to \mfg \) acts on the values of
  \( \omega \).
  In other words,
  \( d_{\mfh}\alpha = - \cA \bigl(\alpha(\omega(\any,\any),
  \any,\dotsc,\any)\bigr) \).
  We thus find the three-form on \( c_{\mfg} \) on \( \mfg \) that is
  \( \cH \)-related to the torsion three-form
  \( c_{\mfh} = -J_{\mfh}^{*}d\sigma_{\mfh} \) on~\( \mfh \) is given
  by
  \begin{equation*}
    c_{\mfg}
    = -J^{*}d_{\mfh}\sigma
    = J^{*}((\xi\hook\sigma) \wedge \omega)
    = - J^{*}\omega \wedge \xi^{\flat}
    = -\cA \bigl(g(J^*\omega(\any,\any),\any)\bigr).
  \end{equation*}
  Now the differential \( dc_{\mfh} \) is \( \cH \)-related to
  \( d_{\mfg}c_{\mfg} \) is given by
  \begin{equation*}
    \begin{split}
      d_{\mfg} c_{\mfg}
      &= - \omega \wedge (\xi\hook c_{\mfg})
        =-\cA(c_{\mfg}(\omega(\any,\any),\any,\any))\\
      &=\frac{1}{3}\cA\Bigl(g\bigl(J^*\omega(\any,\any),
        \omega(\any,\any)\bigr)
        + 2g\bigl(J^*\omega(\omega(\any,\any),\any),\any\bigr)\Bigr).
    \end{split}
  \end{equation*}
  So \( dc_{\mfh} = 0 \) is equation~\eqref{eq:AbelianSKTshear2}.

  Conversely, if \( \mfh \) is two-step solvable, then there is some
  left-invariant shear to \( \mfg = \bR^{2n} \), where
  \( 2n = \dim\mfh \), by~\cite[Proposition~2.5]{FrSw}.
  Now the SKT structure on \( \mfh \) induces is \( \cH \)-related to
  an almost Hermitian structure on~\( \mfg \).
  But \( \mfg \) is Abelian, so this structure is K\"ahler, and the
  structure on \( \mfh \) comes from the above construction.
\end{proof}
\begin{definition}
  Two-step shear data \( (\mfa,\omega) \)
  satisfying~\eqref{eq:AbelianSKTshear1}
  and~\eqref{eq:AbelianSKTshear2} will be called \emph{two-step SKT
  shear data}.
\end{definition}
\begin{notation}
  We set
  \begin{equation*}
    \begin{split}
      \nu_1&:=\nu_1(\omega):=\mathcal{A}\Bigl(g\bigl(J^*\omega(\any,\any), \omega(\any,\any)\bigr)\Bigr)\in \Lambda^4 \mfg^*,\\
      \nu_2&:=\nu_2(\omega):=\mathcal{A}\Bigr(g\bigl(J^*\omega(\omega(\any,\any),\any),\any \bigr)\Bigr)\in \Lambda^4 \mfg^*,\\
      \nu&:=\nu(\omega):=\nu_1(\omega)+2\nu_2(\omega).
    \end{split}
  \end{equation*}
  for two-step shear data \( (\mfa,\omega) \).
  Note that~\eqref{eq:AbelianSKTshear2} holds if and only if
  \( \nu=0 \).
\end{notation}
Summarizing, to find two-step solvable SKT algebras we need to solve
equations~\eqref{eq:Abeliansheardata},~\eqref{eq:AbelianSKTshear1}
and~\eqref{eq:AbelianSKTshear2}.
To do this, let us reformulate the conditions using information from
both the data \( (\mfa,\omega) \) and the K\"ahler structure
\( (g,J,\sigma) \) on \( \mfg=\bR^{2n} \).
For \( V\leq \mfg \) a linear subspace, we write
\begin{equation*}
  V_{J} = V \cap JV \eqand V_{r} = (V_{J})^{\bot} \cap V.
\end{equation*}
We say that \( V \) is \emph{complex} if \( V = V_{J} \) or
\emph{totally real} for \( V = V_{r} \).
Putting \( U:=J \mfa_r\oplus (\mfa\oplus J \mfa_r)^{\perp} \), we thus
have
\begin{equation*}
  \mfg = \mfa \oplus U = \mfa_{J} \oplus \mfa_{r} \oplus U_{J} \oplus U_{r}.
\end{equation*}
We note that
\begin{equation*}
  U_{r}=J\mfa_{r} , U_J= (\mfa\oplus J \mfa_r)^{\perp}.
\end{equation*}
In this setting, it is very natural to decompose also
\( \omega_0 \in (U^{*}\wedge\mfa^{*}) \otimes \mfa \) and
\( \omega_1 \in \Lambda^{2}U^{*} \otimes \mfa \) further, using the
decompositions \( \mfa=\mfa_J\oplus \mfa_r \), \( U=U_J\oplus U_r \)
and the associated dual decompositions.
We then write \( (\omega_0)_{rJ}^J \) for the
\( U_r^*\otimes \mfa_J^*\otimes \mfa_J \)-component of \( \omega_0 \),
etc.
Similarly, we write \( (\omega_1)_{Jr}^r \) for the
\( U_J^*\otimes U_r^*\otimes \mfa_r\cong U_J^*\wedge U_r^*\otimes
\mfa_r \)-part of \( \omega_1 \), etc.

It turns out to be useful to consider the endomorphism
\( (\omega_0)(Y,\any)\in \End(\mfa) \) for \( Y\in U_r \).
More conveniently, we define
\begin{equation}
  \label{eq:definitionofA}
  A_X:=-\omega_0(JX,\any)\in \End(\mfa),\eqcond{for \( X\in \mfa_{r} \),}
\end{equation}
and decompose this further as
\begin{equation}
  \label{eq:definitionocompsofA}
  \begin{split}
    A_X
    &= F_X + G_X + H_X + K_X \\
    &\in \End(\mfa_r) + \Hom(\mfa_J,\mfa_r)
      + \Hom(\mfa_r,\mfa_J) + \End(\mfa_J),
  \end{split}
\end{equation}
so \( F_{X} = -(\omega_0)_{rr}^r(JX,\any) \), etc.
Associated to \( F_{X} \) and \( H_{X} \), we have bilinear maps
\begin{alignat*}{2}
  f&\colon \mfa_r\otimes\mfa_r\to \mfa_r&\qquad f(X,Y)&:=F_X(Y)\\
  h&\colon \mfa_r\otimes\mfa_r\to \mfa_J&h(X,Y)&:=H_X(Y)
\end{alignat*}
for \( X,Y\in\mfa_r \).

\begin{lemma}
  \label{le:Jintegrable}
  The shearing form \( \omega \) satisfies the two-step integrability
  equation~\eqref{eq:AbelianSKTshear1} if and only if
  \begin{enumerate}[(i)]
  \item\label{item:G-zero} \( G_X=0 \) for all \( X\in \mfa_r \),
  \item\label{item:J-KX} \( [J,K_X]=0 \) for all \( X\in\mfa_r \),
  \item\label{item:f-sym}
    \( f\in \mfa_r^*\otimes \mfa_r^*\otimes \mfa_r \) is symmetric in
    its \( \mfa_r^* \)-factors,
  \item\label{item:o1-rrr} \( (\omega_1)_{rr}^r=0 \),
  \item\label{item:o1-rrJ}
    \( J^*(\omega_1)_{rr}^J=-2 J\circ \cA(h) \),
  \item\label{item:o1-Jrr}
    \( (\omega_0)_{Jr}^r=J^*(\omega_1)_{Jr}^r \),
  \item\label{item:o01-J-val}
    \( (\omega_1)_{Jr}^J(JX,JY)-J((\omega_1)_{Jr}^J(X,JY)) =
    (\omega_0)_{Jr}^J(X,Y)+J((\omega_0)_{Jr}^J(JX,Y)) \) for all
    \( X\in U_J \) and all \( Y\in \mfa_r \),
  \item\label{item:o0-JJr} \( (\omega_0)_{JJ}^r \) is of type
    \( (1,1) \) and
    \( J^*(\omega_0)_{JJ}^J=(\omega_0)_{JJ}^J-J(J.(\omega_0)_{JJ}^J)
    \), and
  \item\label{item:o1-JJr} \( (\omega_1)_{JJ}^r \) is of type
    \( (1,1) \) and
    \( J^*(\omega_1)_{JJ}^J=(\omega_1)_{JJ}^J-J(J.(\omega_1)_{JJ}^J)
    \).
  \end{enumerate}
\end{lemma}
\begin{proof}
  First of all, observe that if~\eqref{eq:AbelianSKTshear1} holds for
  \( X,Y\in\mfg=\bR^{2n} \), then this equation is also true when
  evaluated on the pairs \( (JX,Y) \), \( (X,JY) \) and \( (JX,JY) \).
  Surely, all tensors in~\eqref{eq:AbelianSKTshear1} are additionally
  anti-symmetric.
  Moreover, as \( \omega|_{\Lambda^2 \mfa^*}=0 \),
  \eqref{eq:AbelianSKTshear1}~is trivially satisfied for
  \( X,Y\in\mfa_J \), as then both sides of that equation are zero.
  Hence, we are left with the following cases, which correspond
  directly to the conditions of the Lemma:
  \begin{enumerate*}[(a)]
  \item \( X\in \mfa_r \), \( Y\in \mfa_J \) gives~\ref{item:G-zero}
    and~\ref{item:J-KX},
  \item \( X,Y\in \mfa_r \) gives \ref{item:f-sym}, \ref{item:o1-rrr}
    and~\ref{item:o1-rrJ},
  \item \( X\in \mfa_r \), \( Y\in U_J \) leads to \ref{item:o1-Jrr}
    and~\ref{item:o01-J-val},
  \item \( X\in \mfa_J \), \( Y\in U_J \) gives~\ref{item:o0-JJr}, and
    finally
  \item \( X,Y\in U_J \) leads to~\ref{item:o1-JJr}.
  \end{enumerate*}
\end{proof}
Using Lemma~\ref{le:Jintegrable}~\ref{item:G-zero}
and~\ref{item:o1-rrr}, \( \mfa_J\perp (\mfa_r\oplus U_r) \) and
\( U_J\perp \mfa \), one observes

\begin{lemma}
  \label{le:partsofnuequalzero}
  Let \( (\mfa,\omega) \) be two-step shear data.  Then
  \begin{enumerate}[(i)]
  \item \( \nu_1(\omega)=0 \) on
    \( \Lambda^4 \mfa+\Lambda^4(\mfa_J\oplus U_r)+\Lambda^3
    \mfa_J\wedge \mfg \),
  \item \( \nu_2(\omega)=0 \) on
    \( \Lambda^4 \mfa+\Lambda^4(\mfa_J\oplus U_r)+\Lambda^3
    \mfa_r\wedge U_r+\Lambda^4 U_J \), and
  \item \( \nu(\omega)=0 \) on
    \( \Lambda^4 \mfa+\Lambda^4(\mfa_J\oplus U_r) \).
  \end{enumerate}
\end{lemma}

Next, let us look at~\eqref{eq:Abeliansheardata}.
We observe that, evaluated on \( X_1,X_2\in U \), \( Y\in \mfa \), the
equation is satisfied if and only if the endomorphisms
\( \omega_0(X_1,\any) \), \( \omega_0(X_2,\any) \) of \( \mfa \)
commute.
Thus, \eqref{eq:Abeliansheardata}~is true on
\( \Lambda^2 U_r\wedge \mfa \) if and only if for all
\( X_1,X_2\in \mfa_r \) we have \( [A_{X_1},A_{X_2}]=0 \), which is
equivalent to \( [K_{X_1},K_{X_2}]=0 \),
\( K_{X_1} H_{X_2}+H_{X_1} F_{X_2}=K_{X_2} H_{X_1}+H_{X_1} F_{X_1} \)
and \( [F_{X_1},F_{X_2}]=0 \) for all \( X_1,X_2\in\mfa_r \), the
latter equation also may be reformulated, due to \( f \) being
symmetric, as \( f(f(\any,\any),\any) \) being totally symmetric.

The other parts of~\eqref{eq:Abeliansheardata} are more difficult and
we will evaluate them explicitly below in the different cases that we
are going to consider in more detail.
Thus, let us look at~\eqref{eq:AbelianSKTshear2}, i.e.\ at
\( \nu(\omega) \).
Here, the \( \Lambda^2 \mfa_J\wedge \mfa_r\wedge U_r \)-part of
\( \nu(\omega) \) gives equations which solely contain different
components of \( A_X \).
A straightforward computation using the fact that \( f \) is
symmetric, that \( K_X \) commutes with \( J \) and all other
\( K_{X'} \) gives us the second part of

\begin{lemma}
  \label{le:conditionsonFX}
  Let \( \mfa \) be a subspace of \( \bR^{2n} \), let
  \( \omega\in \Lambda^2 (\bR^{2n})^*\otimes \mfa \) with
  \( \omega|_{\Lambda^2 \mfa^*}=0 \)
  fulfilling~\eqref{eq:AbelianSKTshear1} and let \( (g,J) \) be the
  standard K\"ahler structure on \( \bR^{2n} \).
  Then, with the decompositions and notations from above, the two-form
  \( \omega \) satisfies~\eqref{eq:Abeliansheardata} on
  \( \Lambda^2 U_r\wedge \mfa \) and~\eqref{eq:AbelianSKTshear2} on
  \( \Lambda^2 \mfa_J\wedge \mfa_r\wedge U_r \) if and only if for all
  \( X,W\in\mfa_r \) all of the following conditions are valid
  \begin{equation}
    \label{eq:AXcommutes}
    \begin{gathered}
      [K_{X},K_{W}]=0,\qquad [F_{X},F_{W}]=0,\\
      K_{X} H_{W}+H_{X} F_{W}=K_{W} H_{X}+H_{W} F_{X}
    \end{gathered}
  \end{equation}
  and
  \begin{equation}
    \label{eq:F2ndordergsymmetric}
    K_{X}^T K_{W}+K_{X} K_{W}+K_{f(X,W)}
  \end{equation}
  is \( g \)-antisymmetric.
  Here, the superscript ``T'' denotes the \( g \)-transpose of an
  endomorphism of \( \mfa_J \).
\end{lemma}
Before we specialize to the different cases, we will use
\eqref{eq:AXcommutes} and~\eqref{eq:F2ndordergsymmetric} to simplify
the form of the endomorphisms \( K_X \) of \( \mfa_J \) with respect
to a certain basis of \( \mfa_J \).

For this, note that by Lemma~\ref{le:Jintegrable}\ref{item:J-KX} these
endomorphism may be considered as complex endomorphisms on the complex
vector space \( (\mfa_J, J) \).
From now on, we will often consider \( \mfa_J \) as a complex vector
space and note that the scalar multiplication satisfies
\( (a+ib)X=a X+ b JX \) for \( a,b\in \bR \) and \( X\in \mfa_J \).
Even more, we will identify the real Hermitian vector space
\( (\mfa_J,g,J) \) with the complex unitary vector space
\( (\mfa_J,g+i \omega) \).

The rest of this section aims at showing that all complex
endomorphisms \( K_X \) are simultaneously complex diagonalizable by a
unitary basis.
For this, note that Lemma~\ref{le:conditionsonFX} tells us that all
these endomorphisms commute and so we need to show that all these
endomorphisms are diagonalizable.
This will eventually follow from~\eqref{eq:F2ndordergsymmetric} and
for this it turns out to be useful to find a good basis
\( (v_1,\ldots,v_m) \) of \( \mfa_r \) in which \( f(v_i,v_i) \) is
not too complicated.

\begin{lemma}
  \label{le:fspecialbasis}
  Let \( (V,\inp\any\any) \) be an inner product space of
  dimension~\( n \) and let \( f\colon S^2 V\to V \) be linear.
  Then there exists an orthonormal basis \( (v_1,\ldots,v_n) \) of
  \( V \) with \( f(v_i,v_i)\in \spa{v_1,\ldots,v_i} \) for all
  \( i=1,\ldots,n \).
\end{lemma}

\begin{proof}
  The proof is by induction on~\( n \).
  The case \( n=1 \) holds trivially.

  For \( n>1 \), we first show that there is some unit vector
  \( v_1\in V \) with \( f(v_1,v_1)\in \spa{v_1} \).
  If there is some \( v_1\in V\setminus \{0\} \) with
  \( f(v_1,v_1)=0 \), we are done.
  Otherwise, let \( S^{n-1} = \{v\in V: \norm{v}=1\} \) be the unit
  sphere and consider the function \( F\colon S^{n-1}\to S^{n-1} \),
  \( F(v):=f(v,v)/\norm{f(v,v)} \).
  It suffices to show that \( F \)~has a fixed point.
  Since any fixed point-free map from \( S^{n-1} \) to \( S^{n-1} \)
  has degree equal to \( (-1)^n \), it is enough to show that
  \( F \)~has even degree.
  If \( F \) is not surjective, then it is homotopic to a constant map
  and so has degree \( 0 \), and we are done.
  When \( F \) is surjective, Sard's Theorem gives the existence of a
  regular value \( w \) of~\( F \).
  At each point \( v \) of \( F^{-1}(w) \) the differential of \( F \)
  has full rank, so \( F \) is a diffeomorphism in a neighbourhood of
  \( v \).
  This implies that \( F^{-1}(w) \) is discrete in \( S^{n-1} \) and
  so finite.
  But then the local degree \( \deg(F|v) \) of \( F \) at \( v \) is
  \( 1 \) or \( -1 \) and we conclude that
  \( \deg(F) = \sum_{v \in F^{-1}(w)} \deg(F|v) \equiv \lvert
  F^{-1}(w) \rvert \equiv 0\bmod{2} \), since \( v \in F^{-1}(w) \)
  gives \( -v \in F^{-1}(w) \) too.

  So there is some unit vector \( v_1\in V \) with
  \( f(v_1,v_1)\in \spa{v_1} \).
  Then take the orthogonal complement \( W \) of \( \spa{v_1} \)
  in~\( V \).
  Applying the induction hypothesis to the map
  \( h:=\pi_W\circ f|_{S^2 W}\colon S^2 W\to W \), where
  \( \pi_W\colon V\to W \) is orthogonal projection onto~\( W \),
  yields an orthonormal basis \( v_2,\ldots,v_{n} \) with
  \( h(v_i,v_i)\in \spa{v_2,\ldots,v_i} \).
  This gives \( f(v_i,v_i)\in \spa{v_1,\ldots,v_i} \) for all
  \( i=2,\ldots,n \), and the result follows.
\end{proof}
Looking at~\eqref{eq:F2ndordergsymmetric}, one may imagine that the
following lemma will be useful in proving that all \( K_X \) are
complex diagonalizable:

\begin{lemma}
  \label{le:G-skew}
  If \( P \in \bC^{n\times n} \) has \( G(P) = P^2 + P^*P + aP \)
  skew-Hermitian for some \( a \in \bR \), then \( P \) is
  diagonalisable via a unitary transformation and the real parts
  \( \mu \) of the eigenvalues of \( P \) satisfy
  \( \mu(2\mu+a) = 0 \).
\end{lemma}

\begin{proof}
  Write \( P = H + S \), where \( H = (P+P^*)/2 \) is Hermitian,
  \( S= (P-P^*)/2 \) is skew-Hermitian.
  Then \( G(P) = 2H^2 + (HS+SH) + (HS-SH) + aH + aS \).
  Noting that \( HS+SH \) is skew-Hermitian and \( [H,S] = HS-SH \) is
  Hermitian, we see that the Hermitian part of \( G(P) \) is
  \( 2H^2 + aH + [H,S] \).
  By applying a unitary transformation, we may assume that \( H \) is
  diagonal, i.e.\ \( H=\diag(\mu_1,\ldots,\mu_n) \) for certain
  \( \mu_1,\ldots,\mu_n\in \bR \).

  Let now \( G(P) \) be skew-Hermitian.
  Then \( 0=2H^2 + aH + [H,S] \).
  But \( 2H^2 + aH=\diag(\mu_1(2\mu_1+a),\ldots,\mu_n(2\mu_n+a)) \)
  and \( [H,S] \) has no diagonal component.
  Hence, both terms have to vanish individually.
  In particular, \( S \) commutes with \( H \) and so preserves the
  eigenspaces of \( H \).
  But \( S \) is also diagonalisable, with eigenvalues in \( i\bR \),
  and the result follows.
\end{proof}
From this, we get

\begin{proposition}
  \label{pro:formofKX}
  Let \( (\mfa,\omega) \) be two-step shear data such that the shear
  is SKT\@.
  Then there is a unitary basis \( (Y_1,\ldots,Y_m) \) of \( \mfa_J \)
  and \( \alpha_1,\ldots,\alpha_m\in\mfa_r^*\otimes \bC \) such that
  \begin{equation}
    \label{eq:formofKX}
    K_X(Y_i)=\alpha_i(X)\cdot Y_i
  \end{equation}
  for all \( X\in \mfa_r \).
\end{proposition}

\begin{proof}
  Let \( m:=\dim(\mfa_r) \).
  By Lemma~\ref{le:fspecialbasis} and
  Lemma~\ref{le:Jintegrable}\ref{item:f-sym}, there is a basis
  \( (X_1,\ldots,X_m) \) of \( \mfa_r \) with
  \( f(X_i,X_i)\in \spa{X_1,\ldots,X_i} \).
  Set \( K_i:=K_{X_i} \) for \( i=1,\ldots, m \) and note that
  \( K_i \) is a complex endomorphism of the complex vector space
  \( (\mfa_J,J) \) by Lemma~\ref{le:Jintegrable}\ref{item:J-KX}.
  We need to show that \( K_1,\ldots,K_m \) are simultaneously complex
  diagonalizable.
  For this, note that these endomorphisms commute
  by~\eqref{eq:AXcommutes} in Lemma~\ref{le:conditionsonFX}.
  We show inductively that \( K_1,\ldots,K_m \) are simultaneously
  complex diagonalizable for all \( m=1,\ldots,n \).

  For \( m=1 \), this follows directly
  from~\eqref{eq:F2ndordergsymmetric} in
  Lemma~\ref{le:conditionsonFX}, Lemma~\ref{le:G-skew} and
  \( f(X_1,X_1)\in \spa{X_1} \).

  Now let us assume that we have shown that \( K_1,\ldots,K_m \) are
  simultaneously complex diagonalizable for some
  \( s\in \{1,\ldots,m-1\} \).
  As \( K_{m+1} \) commutes with \( K_1,\ldots,K_s \), it preserves
  the common complex eigenspaces of \( K_1,\ldots,K_s \) and so we may
  restrict to such an eigenspace, i.e.\ we may assume that
  \( K_i=\lambda_i \id_{\mfa_J} \) for some \( \lambda_i\in \bC \),
  \( i=1,\ldots,s \).
  Decompose now \( K_i=H_i+S_i \) into its Hermitian part \( H_i \)
  and skew-Hermitian part \( S_i \) for \( i=1,\ldots,m \).
  We may find a unitary basis \( (Y_1,\ldots,Y_k) \) of
  \( (\mfa_J,g,J) \), \( k:=\dim_{\bC}(\mfa_J) \) such that
  \( H_{s+1}=\diag(\mu_1,\ldots,\mu_k) \) for certain
  \( \mu_1,\ldots,\mu_k\in \bR \) with respect to that basis.
  By~\eqref{eq:F2ndordergsymmetric} in Lemma~\ref{le:conditionsonFX},
  the Hermitian part of of
  \( K_{s+1}^* K_{s+1}+K_{s+1} K_{s+1}+\sum_{i=1}^{s+1} a_i K_{i} \)
  is zero, where \( a_1,\ldots,a_{s+1}\in \bR \) are such that
  \( f(X_{s+1},X_{s+1})=\sum_{i=1}^{s+1} a_i X_i \).
  This Hermitian part is given by
  \( 2H_{s+1}^2 + [H_{s+1},S_{s+1}] + \sum_{i=1}^s a_1 \Re(\lambda_i)
  \id_{\mfa_{J}} +a_{s+1} H_{s+1} \).
  Since all terms are diagonal in the unitary basis
  \( (Y_1,\ldots,Y_k) \), apart from \( [H_{s+1},S_{s+1}] \), which is
  off-diagonal, we get \( [H_{s+1},S_{s+1}]=0 \).
  But \( S_{s+1} \) is complex diagonalizable and so \( H_{s+1} \) and
  \( S_{s+1} \) are simultaneously complex diagonalizable.
  Hence, \( K_{s+1}=H_{s+1}+S_{s+1} \) is complex diagonalizable as
  well, which shows the statement for \( s+1 \) and so by induction
  also for \( s=m \).
\end{proof}

It is now tempting to assume that certain parts of these
decompositions vanish and this is what we will now do.

\section{Two-step SKT shears with totally real complement}
\label{sec:UJ=0}

In this case, \( U_{J} = 0 \), \( \omega_0 \) is fully given
by~\eqref{eq:definitionofA} and
\( \omega_1=(\omega_1)_{rr}^J+(\omega_1)_{rr}^r \).
Now by Lemma~\ref{le:Jintegrable}~\ref{item:o1-rrr}
and~\ref{item:o1-rrr}, we have \( (\omega_1)_{rr}^r=0 \) and
\( (\omega_1)_{rr}^J=-2 J \circ \mathcal{A}(h) \).
So here everything is already determined by \( A \) as
in~\eqref{eq:definitionofA}.
If we decompose \( h=h_1+h_2 \) with
\( h_1\in S^2 \mfa_r^*\otimes \mfa_J \),
\( h_2\in \Lambda^2 \mfa_r^*\otimes \mfa_J \), then
\( J^*(\omega_1)_{rr}^J=-2 i h_2 \).

Next, looking at the results in Lemma~\ref{le:Jintegrable} and in
Proposition~\ref{pro:formofKX} and using a complex basis
\( Y_1,\ldots,Y_m \) of \( \mfa_J \) as in
Proposition~\ref{pro:formofKX}, \eqref{eq:AXcommutes} may be rewritten
as \( (f(\any,\any),\any)\in S^3 \mfa_r^*\otimes \mfa_r \) and
\begin{equation*}
  \alpha^i(X_1)h^i(X_2,X_3)-\alpha^i(X_2)h^i(X_1,X_3)=h^i(X_2,f(X_1,X_3))-h^i(X_1,f(X_2,X_3))
\end{equation*}
for all \( X_1,X_2,X_3\in \mfa_r \), where
\( h=\sum_{i=1}^r h^i Y_i \) with
\( h^1,\ldots,h^m\in \mfa_r^*\otimes\mfa_r^*\otimes \bC \).
Imposing these conditions, \eqref{eq:Abeliansheardata} is valid (i.e.\
our shear defines a Lie algebra) if and only if
\( \sum_{j=1}^3 \omega(JX_j,\omega(JX_{j+1},JX_{j+2}))=0 \) for all
\( X_1,X_2,X_3\in \mfa_r \), which is easily seen to be equivalent to
\begin{equation*}
  \alpha_i\wedge h_2^i=0
\end{equation*}
for all \( i=1,\ldots,m \), where we have written
\( h_2=\sum_{i=1}^m h_2^i Y_i \) with
\( h_2^i\in \Lambda^2 \mfa_r^*\otimes \bC \), \( i=1,\ldots,m \).
Note that the latter condition is equivalent to \( \alpha_i=0 \) or
\( h_2^i=\alpha_i\wedge \beta_i \) for some
\( \beta_i\in \mfa_r^*\otimes \bC \).
Summarizing, we need to satisfy the following equations for all
\( i=1,\ldots,m \) and all \( X_1,X_2,X_3\in \mfa_r \):
\begin{equation}
  \label{eq:cubicequations}
  \begin{split}
    \alpha_i(X_1)h^i(X_2,X_3)-\alpha_i(X_2)h^i(X_1,X_2)&=h^i(X_2,f(X_1,X_3))-h^i(X_1,f(X_2,X_3)),\\
    \alpha_i\wedge h_2^i&=0
  \end{split}
\end{equation}

Summarizing, all two-step solvable Lie algebras which are shears of
the Abelian Lie algebra with \( U_J=0 \), i.e.\ which fulfil
\( \mfg=\mfg'+J \mfg' \) if we assume that \( \omega \) is surjective,
are of the following form:

\begin{definition}
  \label{eq:defmfgfhalpha}
  Let \( \mfa_J \) be a \( m \)-dimensional complex vector space with
  basis \( Y_1,\ldots,Y_m \), \( \mfa_r \) be a \( k \)-dimensional
  real vector space.
  Moreover, set \( U_r:=\mfa_r \) and let
  \( I_0\colon \mfa_r\rightarrow U_r \) be a linear isomorphism.
  Next, let \( f\in S^2 \mfa_r^*\otimes \mfa_r \),
  \( h_1\in S^2 \mfa_r^*\otimes \mfa_J \),
  \( h_2\in \Lambda^2 \mfa_r^*\otimes \mfa_J \) and
  \( \alpha_1,\ldots,\alpha_m\in \mfa_r^*\otimes \bC \) be such that
  \( f(f(\any,\any),\any)\in S^3 \mfa_r^*\otimes \mfa_r \) and let
  \( g_0\in S^2 \mfa_r^* \) and \( \tau_0\in \Lambda^2 \mfa_r^* \) be
  given.
  Set \( h:=h_1+h_2 \) and define \( h^i \) and \( h_2^i \) for
  \( i=1,\ldots,m \) as above.

  Then put
  \begin{equation}
    \label{eq:defVSUJ0}
    \mfg:=\mfa_J\oplus \mfa_r\oplus U_r
  \end{equation}
  as vector spaces and endow it with the anti-symmetric bilinear map
  \( [\any,\any] \) uniquely defined by
  \begin{equation}
    \label{def:LieBracketUJ0}
    \begin{split}
      [Y_i,Y_j]&=0,\quad [X,Y_i]=0,\quad [I_0 X,Y_i]=\alpha_i(X)\cdot Y_i,\\
      [I_0X,W]&=f(X,W)+h(X,W),\quad [I_0 X, I_0 W]=-2 i h_2(X,W)
    \end{split}
  \end{equation}
  for \( i,j\in \{1,\ldots,m\} \), \( X,W\in \mfa_r \), extended
  complex linearly if one of the arguments is in~\( \mfa_J \).
  Next, define an almost complex structure \( I \) on \( \mfg \) by
  \begin{equation}
    \label{eq:defJUJ0}
    I(Y_j+X+I_0Z):=i Y_j-Z+I_0 X
  \end{equation}
  for all \( j\in \{1,\ldots,m\} \) and all \( X,Z\in \mfa_r \) and a
  Euclidean metric \( \tilde{g} \) on \( \mfg \) such that
  \( Y_1,\ldots,Y_m,i Y_1,\ldots,i Y_m \) is a
  \( \tilde{g} \)-orthonormal basis of (the realification of)
  \( \mfa_J \), \( \mfa_r\oplus U_r \) is \( \tilde{g} \)-orthogonal
  to \( \mfa_J \) and
  \begin{equation}
    \label{eq:defgUJ0}
    \tilde{g}(X,W)=\tilde{g}(JX,JW)=g_0(X,W),\
    \tilde{g}(JX,W)=-\tilde{g}(X,JW)=\tau_0(X,W)
  \end{equation}
  for all \( X,W\in \mfa_r \).

  We write \( \mfg_{f,h,\alpha} \) for the triple
  \( (\mfg,\tilde{g},I) \) or only for \( \mfg \), where
  \( \alpha:=(\alpha_1,\ldots,\alpha_m) \) and we suppress the
  dependence on the other tensors.
\end{definition}
The discussions above and in the last subsection give us:

\begin{proposition}
  \label{pro:SKTUJ0}
  Let \( \mfg_{f,h,\alpha} \) be as in
  Definition~\ref{def:LieBracketUJ0}.
  If the defining data satisfies~\eqref{eq:cubicequations}, then
  \( \mfg_{f,h,\alpha} \) is an almost Hermitian Lie algebra with
  integrable almost complex structure.
  Moreover, any two-step solvable SKT Lie algebra \( (\mfg,g,J) \)
  with \( \mfg'+J\mfg'=\mfg \) is isomorphic to some
  \( \mfg_{f,h,\alpha} \).
\end{proposition}

Actually, we can say a tiny bit more which will help us classifying
all two-step solvable SKT Lie algebras \( (\mfg,g,J) \) with
\( \mfg'+J\mfg'=\mfg \) and \( \dim(\derg_r)\leq 2 \).

For this, recall that the proof of Lemma~\ref{le:fspecialbasis} yields
some \( v_1\in V\setminus \{0\} \) such that either \( f(v_1,v_1)=0 \)
or \( f(v_1,v_1)=v_1 \).
In fact, the latter case is always true if we additionally have that
\( \im(f)=V \) and \( f(f(\any,\any),\any)\in S^3 V^*\otimes V \).
Even more, \( f(v_1,\any)=\id_V \):

\begin{proposition}
  \label{pro:fidentityendomorphism}
  Let \( V \) be an \( n \)-dimensional real vector space and
  \( f\in S^2 V^*\otimes V \) be such that
  \( f(f(\any,\any),\any)\in S^3 V^*\otimes V \) and \( \im(f) = V \).
  Then there exists a \( v \in V \) with \( f(v,\any) = \id_{V} \).
\end{proposition}

\begin{proof}
  The symmetry assumptions on \( f \) imply that the endomorphisms
  \( f_{v} = f(v, \any) \), \( v \in V \), commute and form a closed
  algebra: \( f_{u} \circ f_{v} = f_{f(u,v)} \).
  It is enough to show that \( \id_{V} \) is in the algebra \( A \)
  which is the complex linear span of \( \{f_{v}:v\in V\} \), since
  \( \id_{V} = \sum_{i=1}^{k} a_{i}f_{v_{i}} \) with
  \( a_{i} \in \bC \) implies
  \( \id_{V} = \sum_{i=1}^{k} \Re(a_{i}) f_{v_{i}} \), as
  \( f_{v_{i}} \) are real.
  Extending \( f \) complex linearly, it is enough to prove the
  assertion on the complexification~\( \Vc \).

  Now the operators \( f_{v} \), \( v \in \Vc \), commute, so
  \( \Vc \) is a direct sum of subspaces~\( W \), on which each
  \( f_{v} \) has only a single eigenvalue.
  The assumption that \( f \) is surjective on the real space~\( V \),
  implies that it is also surjective on \( \Vc \).
  So for a given summand \( W \) there is \( u(w) \in \Vc \) for which
  \( f_{u(w)} \) has a non-zero eigenvalue~\( \lambda \).
  It follows there is a linear combination~\( a \) of these
  \( f_{u(w)} \) that has all eigenvalues non-zero.
  Now consider the characteristic polynomial
  \( \Delta_{a}(t) = \det(a - t\,\id_{V}) = c + p_{a}(t) \), where
  \( c = \det(a) \ne 0 \) and \( p_{a}(t) \) is a polynomial in
  \( t \) without constant term.
  By the Cayley-Hamilton theorem,
  \( 0 = \Delta_{a}(a) = c\,\id_{V} + p_{a}(a) \), so
  \( \id_{V} = -p_{a}(a)/c \) which is an element of the
  algebra~\( A \) and hence has the form \( f_{w} \) for some
  \( w \in \Vc \).
  By the remarks above we have \( \id_{V} = f_{v} \) for
  \( v = \Re(w) \), as required.
\end{proof}
Proposition~\ref{pro:fidentityendomorphism} yields:

\begin{corollary}
  \label{co:derg+Jderg=g}
  Let \( (\mfg,g,J) \) be a two-step solvable Lie algebra with
  \( \derg+J\derg=\mfg \).
  Then there exist some \( X\in \derg_r \) and
  \( h_1,h_2\in \End(\derg_r,\derg_J) \) such that for all
  \( W\in \derg_r \) we have
  \begin{equation*} [JX,W]=W+h_1(W),\qquad [JW,X]=W+h_2(W).
  \end{equation*}
\end{corollary}

Note that Proposition~\ref{pro:fidentityendomorphism} does not give us
any information on the endomorphism \( f(w,\any) \) for
\( w\notin \spa{v} \) besides \( f(w,v)=w \).
This is one of the reasons why we are restricting in the following to
the case \( \dim(\derg_r)\leq 2 \).
Another reason is that surely~\eqref{eq:cubicequations} gets easier in
this case, in particular, the second equation
in~\eqref{eq:cubicequations} is void in this case.

\subsection{Almost Abelian algebras}

An \emph{almost Abelian} Lie algebra is one with a codimension one
Abelian ideal and the SKT structures on such Lie algebras of dimension
\( 2n \) are precisely those which may be obtained by a shear from the
flat K\"ahler SKT Lie algebra \( (\bR^{2n},g_0,J_0,\sigma_0) \) with
\( \dim(\mfa)=2n-1 \) or, equivalently, with \( U_J=0 \) and
\( \dim(\mfa_r)=1 \).
Note that a characterization of those almost Abelian Lie algebras
which are SKT has already been given in~\cite{AL}, cf.\ also~\cite{FP}
for a more detailed list in the six-dimensional case.

We prove this characterization by a different approach and even solve
the remaining equations in~\cite{AL} completely.
Note further that, in contrast to all other cases in this paper, we
will \emph{not} assume here that \( \im(\omega)=\mfa \), i.e.\ we also
do not necessarily have \( \mfg'+J\mfg'=\mfg \).
Still, one easily sees that we may use all the results from above as
so far we have never used that \( \im(\omega)=\mfa \).

Choose \( X\in \mfa_r \) of norm one.
Then \( f(X,X)=a X \) for some \( a\in \mfa_r \) and
\( h(X,X)=\sum_{i=1}^{n-1} w_i Y_i \) for certain
\( w_1,\ldots,w_{n-1}\in \bC \).
Moreover, \( \alpha_i(X)=z_i \) for \( z_i\in \bC \),
\( i=1,\ldots,n-1 \).
Note that here also the first equation in~\eqref{eq:cubicequations} is
void in total and so we only still need to
impose~\eqref{eq:AbelianSKTshear2}, i.e.\ \( \nu(\omega)=0 \).
But by Lemma~\ref{le:partsofnuequalzero} and for dimensional reasons,
\( \nu(\omega) \) only has a component in
\( \Lambda^2 \mfa_J^*\wedge \mfa_r^*\wedge U_r^* \).
The corresponding equation is given in~\eqref{eq:F2ndordergsymmetric}
and is equivalent to
\begin{equation}
  \label{eq:z-a}
  \overline{z}_i z_i+z_i^2+a z_i=-\overline{\overline{z}_i
  z_i+z_i^2+a z_i}
\end{equation}
This may be written as \( \Re(z_i)(2\Re(z_i)+a)=0 \).
So \( \Re(z_i)\in \{0,-a/2\} \) for any \( i=1,\ldots,n-1 \).
Thus, we have obtained:

\begin{theorem}
  \label{th:AlmostAbelianSKTLAs}
  Let \( (\mfg,g,J) \) be a \( 2n \)-dimensional almost Abelian SKT
  Lie algebra.
  Then \( (\mfg,g,J) \) is isomorphic to \( \mfg_{f,h,\alpha} \) with
  \( \mfa_r=\spa{X} \) for a unit vector \( X \),
  \( h(X,X)=\sum_{i=1}^{n-1} w_i Y_i \), \( f(X,X)=a X \) and
  \( \alpha(X)=\sum_{i=1}^{n-1} z_i Y_i \) for certain \( a\in \bR \),
  \( z=(z_1,\ldots,z_{n-1}) \),
  \( w=(w_1,\ldots,w_{n-1})\in \bC^{n-1} \) with
  \( \Re(z_i)\in \{0,-a/2\} \) for \( i=1,\ldots, n-1 \).
  Conversely, all such \( \mfg_{f,h,\alpha} \) are almost Abelian SKT
  Lie algebras.
\end{theorem}

\begin{remark}
  \label{re:LiebracketalmostAbelianSKT}
  So any almost Abelian SKT Lie algebra \( (\mfg,g,J) \) with
  codimension one Abelian ideal \( \mfa \) has a real basis
  \( Y_1,\ldots,Y_{n-1}, i Y_1,\ldots,iY_{n-1},X,JX \), where
  \( Y_1,\ldots,Y_{n-1} \) is a complex basis of \( \mfa_J \) and
  \( X \) spans \( \mfa_r \), such that
  \( Y_1,\ldots,Y_{n-1}, i Y_1,\allowbreak\ldots,iY_{n-1},X \) is an
  orthonormal system and there is some \( m\in \{0,\ldots,n-1\} \)
  such the only non-zero Lie brackets (up to anti-symmetry and
  complex-linear extension of the bracket to \( \mfa_J \)) are given
  by
  \begin{equation*}
    [JX,Y_j]=\Bigl(-\frac{a}{2}+i b_j\Bigr)Y_j,\quad
    [JX,Y_k]=i b_k Y_k, \quad [JX,X]=a X+ \sum_{i=1}^m w_i Y_i
  \end{equation*}
  for \( j=1,\ldots,m \), \( k=m+1,\ldots,n-1 \) and certain
  \( a\in \bR \), \( b=(b_1,\ldots,b_n)\in \bR^{n-1} \),
  \( w=(w_1,\ldots,w_{n-1})\in \bC^{n-1} \).
\end{remark}

Since two almost Abelian Lie algebras
\( \bR^{2n-1}\rtimes_{f_1} \bR \) and
\( \bR^{2n-1}\rtimes_{f_2} \bR \) are isomorphic if and only if
\( f_1 \) is conjugated to \( f_2 \) up to a non-zero multiple,
Theorem \ref{th:AlmostAbelianSKTLAs} implies

\begin{corollary}
  \label{co:almostAbelian}
  A \( 2n \)-dimensional almost Abelian Lie algebra
  \( \bR^{2n-1}\rtimes_f \bR \) admits an SKT structure if and only if
  \begin{itemize}
  \item[(i)] either \( f \) is diagonalizable over the complex numbers
    and there is one real eigenvalue \( a\in \bR \) such that the
    other complex eigenvalues come \( n-1 \) pairs
    \( (z_i,\overline{z}_i) \) with
    \( \Re(z_i)\in \{0,-\frac{a}{2}\} \), \( i=1,\ldots,n-1 \),
  \item[(ii)] or the complex Jordan normal form of \( f \) has one
    Jordan block of size \( 2 \) with eigenvalue \( 0 \) and all other
    Jordan block are of size \( 1 \) with eigenvalues
    \( i b_1,\allowbreak-i b_1,\ldots, i b_{n-2},-i b_{n-2},0 \) for
    certain real numbers \( b_1,\ldots,b_{n-2}\in\bR \).
  \end{itemize}
\end{corollary}

In the six-dimensional case, Corollary~\ref{co:almostAbelian} yields,
in agreement with~\cite{FP}, that the following almost Abelian Lie
algebras admit an SKT-structure, where we refer to
Table~\ref{table_LAs} for the Lie brackets of the mentioned Lie
algebras:

\begin{corollary}
  \label{co:6dalmostAbelianSKTLAs}
  Let \( \mfg \) be a six-dimensional almost Abelian Lie algebra.
  Then \( \mfg \) admits an SKT structure if and only if \( \mfg \) is
  isomorphic to one of the following Lie algebras:
  \begin{equation*}
    \begin{gathered}
      \mfg_{6,1}^{-\tfrac{1}{2},-\tfrac{1}{2},-\tfrac{1}{2},-\tfrac{1}{2}},\
      \mfg_{6,8}^{a,-\tfrac{a}{2},-\tfrac{a}{2},-\tfrac{a}{2}},\
      \mfg_{6,8}^{a,-\tfrac{a}{2},-\tfrac{a}{2},0},\
      \mfg_{6,11}^{a,-\tfrac{a}{2},-\tfrac{a}{2},s},\
      \mfg_{6,11}^{a,0,-\tfrac{a}{2},s},\\
      \mfg_{6,11}^{a,-\tfrac{a}{2},0,s},\ \mfg_{6,11}^{a,0,0,s},\
      \mfg_{5,14}^0\oplus
      \bR,\  \mfg_{5,17}^{0,0,s}\oplus \bR,\ \mfr_{4,-1/2,-1/2} \oplus \bR^2,\\
      \mfr_{4,a,-\tfrac{a}{2}}'\oplus \bR^2\:(a>0),\ \mfr_{4,a,0}'\oplus
      \bR^2\:(a>0),\ \mfr_{3,0}'\oplus \bR^3, \ \mfh_3\oplus \bR^3,\
      \aff_{\bR} \oplus \bR^4,\ \bR^6.
    \end{gathered}
  \end{equation*}
\end{corollary}

\subsection{Commutator ideal not complex and of
codimension two}

Note that if \( \mfg' \) is of codimension two, the condition
\( \mfg'+J\mfg'=\mfg \) is equivalent to \( \mfg' \) being
non-complex.

By Proposition~\ref{pro:fidentityendomorphism}, we may choose an
orthonormal basis \( X_1,X_2 \) of \( \mfa_r \) such that
\( f(X_1,\any)=a\id_{\mfa_r} \) for some \( a>0 \) and such that
\( f(X_2,X_2)= b_1 X_1+b_2 X_2 \) for certain \( b_1,b_2\in \bR \)
with \( b_2\geq 0 \).

We set \( z_i:=\alpha_i(X_1), w_i:=\alpha_i(X_2)\in \bC \) for
\( i=1,\ldots,n-2 \) and first look at~\eqref{eq:F2ndordergsymmetric}.
As in the last subsection, this equation, applied to \( X=W=X_1 \),
gives us \( \Re(z_i)\in \{0,-a/2\} \).
Next, applying~\eqref{eq:F2ndordergsymmetric} to \( X=X_1,W=X_2 \), we
obtain
\begin{equation*}
  \Re(w_i)(2\Re(z_i)+a)=0,
\end{equation*}
so that \( \Re(w_i)=0 \) if \( \Re(z_i)=0 \).
Otherwise, the above equation is automatically fulfilled.
Finally, if \( X=W=X_2 \), \eqref{eq:F2ndordergsymmetric}~yields
\begin{equation*}
  2 \Re(w_i)^2+b_2 \Re(w_i)+b_1 \Re(z_i)=0.
\end{equation*}
If \( \Re(z_i)=0 \) and so also \( \Re(z_i)=0 \), this equation is
automatically satisfied.
Otherwise, \( \Re(w_i) \) has to be a solution of the quadratic
equation
\begin{equation*}
  2 x^2+b_2 x-\frac{1}{2} b_1 a=0,
\end{equation*}
which has a solution if \( b_2^2+4 b_1 a\geq 0 \).
So our parameters \( a,b_1,b_2 \) have to fulfil this inequality if
for some \( i\in \{1,\ldots,n-2\} \) we have \( \Re(z_i)\neq 0 \).

Next, we look at the first equation in~\eqref{eq:cubicequations},
which yields
\begin{equation*}
  \begin{split}
    z_i h^i(X_2,X_1)-w_i h^i(X_1,X_1)&=a h^i(X_2,X_1)-a h^i(X_1,X_2),\\
    z_i h^i(X_2,X_2)-w_i h^i(X_1,X_2)&=a h^i(X_2,X_2)-b_1 h^i(X_1,X_1)-b_2 h^i(X_1,X_2).
  \end{split}
\end{equation*}
Since \( z_i-a\neq 0 \) due to \( \Re(z_i)\in \{0,-a/2\} \) and
\( a\neq 0 \), these equations are equivalent to
\begin{equation}
  \label{eq:hi}
  \begin{split}
    h^i(X_2,X_1)&=\frac{w_i}{z_i-a} h^i(X_1,X_1)-\frac{a}{z_i-a} h^i(X_1,X_2),\\
    h^i(X_2,X_2)&=-\frac{b_1}{z_i-a} h^i(X_1,X_1)+\frac{w_i-b_2}{z_i-a} h^i(X_1,X_2).
  \end{split}
\end{equation}
Summarizing the discussion up to this point, we have obtained:

\begin{lemma}
  \label{le:codim2notJinv}
  Let \( (\mfg,g,J) \) be a two-step solvable SKT Lie algebra with
  \( \mfg' \) being non-complex and of codimension two.
  Then \( (\mfg,g,J) \) is isomorphic to \( \mfg_{f,h,\alpha} \) with
  \( X_1,X_2 \) being an orthonormal basis of \( \mfa_r \), \( f \)
  satisfying
  \begin{equation}
    \label{eq:conditionsonfcodim2case}
    f(X_1,X_1)=a X_1,\quad f(X_1,X_2)=f(X_2,X_1)=a X_2,\quad f(X_2,X_2)=b_1 X_1+b_2 X_2
  \end{equation}
  for certain \( a>0 \), \( b_2\geq 0 \), \( h \)
  fulfilling~\eqref{eq:hi} and for all \( i=1,\ldots,n-2 \), and
  \( z_i,w_i\in \bC \) defined by
  \begin{equation*}
    \alpha_i(X_1)=z_i,\quad\alpha(X_2)=w_i
  \end{equation*}
  satisfy either
  \begin{enumerate}[(i)]
  \item \( \Re(z_i)=\Re(w_i)=0 \), or
  \item \( \Re(z_i)=-a/2 \) and
    \( 2\Re(w_i)^2+b_2 \Re(w_i)-\frac{1}{2} b_1 a=0 \), where this
    case may only occur if \( b_1\geq -\tfrac{b_2^2}{4a} \).
  \end{enumerate}
  Moreover, with this data, \( \mfg_{f,h,\alpha} \) is an almost
  Hermitian Lie algebra with integrable almost complex structure
  \( I \).
\end{lemma}

We still need to solve~\eqref{eq:AbelianSKTshear2} and remark that we
already solved it on \( \Lambda^2 \mfa_J\wedge \mfa_r\wedge U_r \).
Using Lemma~\ref{le:partsofnuequalzero} and for dimensional reasons,
we only have to consider this equation on
\( \Lambda^2 \mfa_r\wedge \Lambda^2 U_r \),
\( \mfa_J\wedge \Lambda^2 \mfa_r\wedge U_r \) and
\( \mfa_J\wedge \mfa_r\wedge \Lambda^2 U_r \).
We first look at the latter two equations.
Let \( Y\in \mfa_J \) and \( k\in \{1,2\} \) be given.
Since \( J^*\nu_1(\omega)=\nu_1(\omega) \), we have
\( \nu_1(\omega)(Y,X_1,X_2,JX_k)=-\nu_1(\omega)(iY,JX_1,JX_2,X_k) \)
and so \( \nu(\omega)(Y,X_1,X_2,JX_k)=0 \) and
\( \nu(\omega)(iY,X_k,JX_1,JX_2)=0 \) are equivalent to
\begin{equation*}
  \begin{split}
    \nu_1(Y,X_1,X_2,JX_k)+2\nu_2(Y,X_1,X_2, JX_k)&=0,\\
    \nu_2(Y,X_1,X_2,JX_k)&=-\nu_2(iY,JX_1,JX_2,X_k).
  \end{split}
\end{equation*}
In the following, we use that \( \omega(\mfa,\mfa)=0 \),
\( \omega(U_r,\mfa_J)\subseteq \mfa_J \),
\( \omega(U_r,U_r)\subseteq \mfa_J \), \( \mfa_J\perp \mfa_r \) and
the fact that \( f \) is symmetric and
\( f(f(\any,\any),\any)\in S^3 \mfa_r^*\otimes \mfa_r \) to simplify
these equations and first look at the second equation.  We have
\begin{equation*}
  \begin{split}
    \nu_1(\omega)(X_1,X_2,JX_k,Y)=\tfrac{1}{12}\, & g\left(-h(f(X_k,X_1),X_2)+h(X_2,f(X_k,X_1))\right.\\
                                                  &\quad+K_{X_2}(h(X_k,X_1))+h(f(X_k,X_2),X_1)\\
                                                  &\quad \left.-h(X_1,f(X_k,X_2))-K_{X_1}(h(X_k,X_2)),iY\right)
  \end{split}
\end{equation*}
and
\begin{equation*}
  \begin{split}
    \MoveEqLeft
    -\nu_2(\omega)(JX_1,JX_2,X_k,iY)\\
    &=\frac{1}{12}\,g\bigl(K_{X_k}(h(X_2,X_1)-h(X_1,X_2)) \eqbreak[4]
      +h(f(X_k,X_2),X_1)-h(f(X_k,X_1),X_2),iY\bigr),
  \end{split}
\end{equation*}
i.e.\
\( \nu_2(\omega)(Y,X_1,X_2,JX_k)=-\nu_2(\omega)(iY,JX_1,JX_2,X_k) \)
for all \( Y\in \mfa_J \) is equivalent to
\begin{equation*}
  \begin{split}
    &-h(f(X_k,X_1),X_2)+h(X_2,f(X_k,X_1))+K_{X_2}(h(X_k,X_1))+h(f(X_k,X_2),X_1)\\
    &-h(X_1,f(X_k,X_2))-K_{X_1}(h(X_k,X_2))\\
    &=K_{X_k}(h(X_2,X_1)-h(X_1,X_2))+h(f(X_k,X_2),X_1)-h(f(X_k,X_1),X_2).
  \end{split}
\end{equation*}
However, by~\eqref{eq:AXcommutes}, we have
\begin{equation*}
  \begin{split}
    \MoveEqLeft
    -h(f(X_k,X_1),X_2)+h(X_2,f(X_k,X_1))+K_{X_2}(h(X_k,X_1)) \eqbreak
    + h(f(X_k,X_2),X_1) -h(X_1,f(X_k,X_2))-K_{X_1}(h(X_k,X_2))\\
    &=-h(f(X_k,X_1),X_2)+h(X_k,f(X_2,X_1))+K_{X_k}(h(X_2,X_1)) \eqbreak
      +h(f(X_k,X_2),X_1) -h(X_k,f(X_1,X_2))-K_{X_k}(h(X_1,X_2))\\
    &=K_{X_k}(h(X_2,X_1)-h(X_1,X_2))+h(f(X_k,X_2),X_1)-h(f(X_k,X_1),X_2),
  \end{split}
\end{equation*}
i.e
\( \nu_2(\omega)(Y,X_1,X_2,JX_k)=-\nu_2(\omega)(iY,JX_1,JX_2,X_k) \)
is automatically satisfied.

Next, we consider
\( \nu_1(Y,X_1,X_2,JX_k)+2\nu_2(Y,X_1,X_2, JX_k)=0 \) and note that
\begin{equation*}
  \begin{split}
    6\,\nu_1(X_1,X_2,JX_k,Y)
    &= g(K_{X_2}^T(h(X_k,X_1))-K_{X_1}^T(h(X_k,X_2)) \eqbreak[4]
      +K_{X_k}^T(h(X_2,X_1)-h(X_1,X_2)),iY).
  \end{split}
\end{equation*}
Thus, using that
\( \nu_2(X_1,X_2,JX_k,Y)=-\nu_2(\omega)(JX_1,JX_2,X_k,iY) \), one sees
that \( \nu_1(X_1,X_2,JX_k,Y)+2\nu_2(X_1,X_2, JX_k,Y)=0 \) for all
\( Y\in \mfa_J \) is equivalent to
\begin{equation*}
  \begin{split}
    0&=(K_{X_k}+K_{X_k}^T)(h(X_2,X_1)-h(X_1,X_2))
       +h(f(X_k,X_2),X_1)        \eqbreak
       -h(f(X_k,X_1),X_2)
       +K_{X_2}^T(h(X_k,X_1))-K_{X_1}^T(h(X_k,X_2)).
  \end{split}
\end{equation*}
We look at these equations componentwise and for \( k=1,2 \).
For \( k=1 \), we obtain
\begin{equation}
  \label{eq:hik=1}
  0=(2\Re(z_i)+a)\bigl(h^i(X_2,X_1)-h^i(X_1,X_2)\bigr)
  -\overline{z}_i\, h^i(X_1,X_2)+\overline{w}_i\, h^i(X_1,X_1),
\end{equation}
whereas for \( k=2 \), we get
\begin{equation}
  \label{eq:hik=2}
  \begin{split}
    0&=b_1 h^i(X_1,X_1)-2\Re(w_i)h^i(X_1,X_2) \eqbreak
       +(2\Re(w_i)+b_2+\overline{w}_i) h^i(X_2,X_1)-(\overline{z}_i+a)
       h^i(X_2,X_2).
  \end{split}
\end{equation}
Now either \( \Re(z_i)=0 \) or \( \Re(z_i)=-{a}/{2} \).

In the first case, \eqref{eq:hik=1}~gives us
\begin{equation*}
  \begin{split}
    h^i(X_2,X_1)
    &=-\frac{\overline{w}_i}{a}
      h^i(X_1,X_1)+\frac{\overline{z}_i+a}{a}h^i(X_1,X_2)\\
    &=\frac{w_i}{a} h^i(X_1,X_1)+\frac{a-z_i}{a}h^i(X_1,X_2)
  \end{split}
\end{equation*}
Together with~\eqref{eq:hi}, we obtain
\begin{equation*}
  \frac{w_i}{z_i-a} h^i(X_1,X_1)-\frac{a}{z_i-a} h^i(X_1,X_2)=\frac{w_i}{a} h^i(X_1,X_1)+\frac{a-z_i}{a}h^i(X_1,X_2)
\end{equation*}
i.e.\
\begin{equation*}
  z_i h^i(X_1,X_2)= w_i h^i(X_1,X_1)
\end{equation*}
where we have used that \( 2a-z_i\neq 0 \).

So we have to discuss the two subcases \( z_i\neq 0 \) and
\( z_i=0 \).
First for \( z_i\neq 0 \): Here we get
\( h^i(X_2,X_1)=\frac{w_i}{z_i} h^i(X_1,X_1) \) and so, inserting
into~\eqref{eq:hi}, we obtain
\begin{equation*}
  h^i(X_2,X_1)=\frac{w_i}{z_i} h^i(X_1,X_1)=h^i(X_1,X_2)
\end{equation*}
and
\begin{equation*}
  h^i(X_2,X_2)=\frac{w_i^2-b_2 w_i-b_1 z_i}{z_i(z_i-a)} h^i(X_1,X_1),
\end{equation*}
and one checks that \eqref{eq:hik=1} and~\eqref{eq:hik=2} are also
satisfied then.

Second, for \( z_i=0 \): Then \( w_i h^i(X_1,X_1)=0 \) and the first
equation in~\eqref{eq:hi} gives us \( h^i(X_2,X_1)=h^i(X_1,X_2) \) and
\eqref{eq:hik=1} is automatically satisfied.
We now show \( w_{i}\ne0 \), because \( \im(\omega)=\mfa \).

Suppose \( w_i=0 \).
Since \( \omega(JX_1,Y_i)=-z_i Y_i \), \( \omega(JX_2,Y_i)=-w_i Y_i \)
for all \( i=1,\ldots,n-2 \), we may ignore those \( i \) where
\( z_i\neq 0 \) or \( w_i\neq 0 \), i.e.\ we may assume that
\( \omega(JX_i,\mfa_J)=0 \) globally for all \( i=1,2 \) and
\( \dim_{\bR}(\mfa_J)\geq 2 \).
But then we have \( \omega(JX_1,X_2)=\omega(JX_2,X_1) \) and so
\( \omega(JX_1,JX_2)=0 \).
Hence, \( \omega(JX_1,X_1) \), \( \omega(JX_1,X_2) \),
\( \omega(JX_2,X_2) \) have to span the at least four-dimensional real
vector space \( \mfa=\mfa_J\oplus \mfa_r \), a contradiction.

Thus, \( w_i\neq 0 \) and so we must have \( h^i(X_1,X_1)=0 \).
But then \( h^i(X_2,X_2)=\frac{b_2-w_i}{a} h^i(X_1,X_2) \) and again
\eqref{eq:hik=2} is satisfied.

\smallbreak Finally, we need to consider the case
\( \Re(z_i)=-{a}/{2} \), i.e.\ \( 2\Re(z_i)+a=0 \).
Then \eqref{eq:hik=1} is equivalent to
\( h^i(X_1,X_2)=\frac{\overline{w_i}}{\overline{z}_i} h^i(X_1,X_1) \)
and the first equation in~\eqref{eq:hi} yields
\begin{equation*}
  h^i(X_2,X_1)=\frac{w_i \overline{z}_i-a \overline{w}_i}{\abs{z_i}^2-a\overline{z}_i} h^i(X_1,X_1)
\end{equation*}
whereas the second equation in~\eqref{eq:hi} gives us
\begin{equation*}
  h^i(X_2,X_2)=\frac{\abs{w_i}^2-b_1 \overline{z}_i-b_2 \overline{w}_i}{\abs{z_i}^2-a\overline{z}_i} h^i(X_1,X_1)
\end{equation*}
Finally, we need to consider~\eqref{eq:AbelianSKTshear2} on
\( \Lambda^2 \mfa_r\wedge \Lambda^2 U_r \).  Here, we get
\begin{equation*}
  \begin{split}
    \MoveEqLeft
    6 \nu_1(\omega)(X_1,X_2,JX_1,JX_2)\\
    &=2 a(a-b_1)+\norm{h(X_1,X_2)-h(X_2,X_1)}^2+\norm{h(X_1,X_2)}^2 \eqbreak
      +\norm{h(X_2,X_1)}^2-2 g(h(X_1,X_1),h(X_1,X_2)),
  \end{split}
\end{equation*}
whereas
\begin{equation*}
  \begin{split}
    \MoveEqLeft
    12\, \nu_1(\omega)(X_1,X_2,JX_1,JX_2) \\
    &=-g(\omega(J\omega(X_1,JX_1),X_2),X_2)
      +g(\omega(J\omega(X_1,JX_2),X_1),X_2) \eqbreak
      +g(\omega(J\omega(X_2,JX_1),X_2),X_1)
      -g(\omega(J\omega(X_2,JX_2),X_1),X_1)\\
    &=g(f(f(X_1,X_1),X_2),X_2)-g(f(f(X_2,X_1),X_1),X_2) \eqbreak
      -g(f(f(X_1,X_2),X_2),X_1)+g(f(f(X_2,X_2),X_1),X_1)\\
    &=0,
  \end{split}
\end{equation*}
since \( f(f(\any,\any),\any) \) is totally symmetric in its
arguments.
Hence, \eqref{eq:AbelianSKTshear2} is true on
\( \Lambda^2 \mfa_r\wedge \Lambda^2 U_r \) if and only if
\begin{equation}
  \label{eq:unsolvableeqnUJ=0codim2}
  \begin{multlined}
    2 a(a-b_1)+\norm{h(X_1,X_2)-h(X_2,X_1)}^2+\norm{h(X_1,X_2)}^2 \\
    +\norm{h(X_2,X_1)}^2 -2 g(h(X_1,X_1),h(X_1,X_2))=0.
  \end{multlined}
\end{equation}
Unfortunately, \eqref{eq:unsolvableeqnUJ=0codim2} seems not to be
solvable in a nice and short form in general.

Summarizing, we have obtained:

\begin{theorem}
  \label{th:codim2notJinv}
  Let \( (\mfg,g,J) \) be a two-step solvable \( 2n \)-dimensional SKT
  Lie algebra with \( \mfg' \) being non-complex and of codimension
  two.
  Then \( (\mfg,g,J) \) is isomorphic to \( \mfg_{f,h,\alpha} \) with
  \( (X_1,X_2) \) being an orthonormal basis of \( \mfa_r \),
  \( f(X_1,X_j)=a X_j \) for \( j=1,2 \) and
  \( f(X_2,X_2)= b_1 X_1+b_2 X_2 \) for certain \( a>0 \),
  \( b_1\in \bR \), \( b_2\geq 0 \) and
  \( \alpha(X_1)=\sum_{i=1}^{n-2} z_i Y_i \),
  \( \alpha(X_2)=\sum_{i=1}^{n-2} w_i Y_i \) for
  \( (z_1,\ldots,z_{n-2}),(w_1,\ldots,w_{n-2})\in \bC^{n-2} \) such
  that \eqref{eq:unsolvableeqnUJ=0codim2} is valid and for all
  \( i=1,\ldots,n-2 \) we have one of the following three cases:
  \begin{enumerate}[(i)]
  \item\label{item:2dr-rz0} \( \Re(z_i)=0 \), \( \Re(w_i)=0 \),
    \( z_i\neq 0 \) and
    \begin{equation*}
      \begin{split}
        h^i(X_2,X_1)&=h^i(X_1,X_2)=\frac{w_i}{z_i} h^i(X_1,X_1),\\
        h^i(X_2,X_2)&=\frac{w_i^2-b_2 w_i-b_1 z_i}{z_i(z_i-a)} h^i(X_1,X_1),
      \end{split}
    \end{equation*}
  \item\label{item:2dr-z0} \( z_i=0 \), \( \Re(w_i)=0 \),
    \( w_i\neq 0 \) and
    \begin{equation*}
      h^i(X_1,X_1)=0,\quad
      h^i(X_2,X_1)=h^i(X_1,X_2), \quad h^i(X_2,X_2)=\frac{b_2-w_i}{a}
      h^i(X_1,X_2),
    \end{equation*}
  \item\label{item:2dr-rez-a} \( \Re(z_i)=-a/2 \), \( \Re(w_i) \) is a
    solution of \( 2x^2+b_2 x-\tfrac{b_1}{2} a=0 \) and
    \begin{equation*}
      \begin{split}
        h^i(X_2,X_1)&=\frac{w_i \overline{z}_i-a
                      \overline{w}_i}{\abs{z_i}^2-a\overline{z}_i}
                      h^i(X_1,X_1),\\
        h^i(X_2,X_2)&=\frac{\abs{w_i}^2-b_1 \overline{z}_i-b_2
                      \overline{w}_i}{\abs{z_i}^2-a\overline{z}_i}
                      h^i(X_1,X_1),
      \end{split}
    \end{equation*}
    where this case may only occur if \( b_1\geq -{b_2^2}/(4a) \).
  \end{enumerate}
  Conversely, all such \( \mfg_{f,h,\alpha} \) are two-step solvable
  SKT Lie algebras with \( \mfg' \) being non-complex and of
  codimension two.
\end{theorem}

\begin{remark}
  In Theorem~\ref{th:codim2notJinv}, we may always choose
  \( h\equiv 0 \).
  Then \eqref{eq:unsolvableeqnUJ=0codim2} is valid if and only if
  \( b_1=a \).
  But then surely also \( a=b_1\geq -{b_2^2}/(4a) \) if and only if
  \( 4 a^2\geq -b_2^2 \) and so also case~\ref{item:2dr-rez-a} may
  occur.
  Moreover, setting \( b:=b_2 \).
  the solutions of \( 2x^2+bx-(a^2/2)=2x^2+b_2 x-(b_1/2) a=0 \) are
  given by \( x=(-b\pm \sqrt{b^2+4a^2})/4 \).

  In this case, there are \( 0\leq r_1\leq r_2\leq r_3\leq n-1 \),
  \( c_{r_1},\ldots,c_{n-1} \), \( d_1,\ldots,d_{n-1} \) such that the
  only non-zero Lie brackets (up to anti-symmetry and complex-linear
  extension to \( \mfa_J \)) are given by
  \begin{equation*}
    \begin{split}
      [JX_2,Y_j]&= i\, d_j Y_j,\ j=1,\ldots,r_1,\\
      [JX_1,Y_k]&= i\, c_k Y_k,\quad [JX_2,Y_k]= i\, d_k Y_k,\quad
                  k=r_1+1,\ldots,r_2,\\
      [JX_1,Y_\ell ]&=\left(-\tfrac{a}{2}+i\, c_\ell \right) Y_\ell ,\quad
                      [JX_2,Y_\ell ]= \left(\tfrac{-b+
                      \sqrt{b^2+4a^2}}{4}+i\, d_\ell \right)
                      Y_\ell , \eqbreak[20] \ell =r_2+1,\ldots,r_3,\\
      [JX_1,Y_m]&=\left(-\tfrac{a}{2}+i\, c_m\right) Y_m,\quad [JX_2,Y_m]=
                  \left(\tfrac{-b- \sqrt{b^2+4a^2}}{4}+i\, d_m\right)
                  Y_m,\eqbreak[20] m=r_2+1,\ldots,r_3,\\
      [JX_1,X_1]&=a X_1,\  [JX_1,X_2]= a X_2,\  [JX_2,X_1]= a X_2,\
                  [JX_2,X_2]=a X_1+b X_2.
    \end{split}
  \end{equation*}
\end{remark}

\section{Commutator ideal totally real}
\label{sec:totallyreal}

Here, we look at two-step SKT shear data \( (\mfa,\omega) \) with
totally real \( \mfa \), i.e.\ \( \mfa=\mfa_r \) and \( \mfa_J=0 \),
and with \( \im(\omega)=\mfa \).
We note that then the shear Lie algebra \( \mfg \) has a totally real
commutator ideal \( \derg \), i.e.\ we have
\( \derg\cap J\derg=\{0\} \).

\begin{notation}
  We simplify the notation and, for \( Y\in U_J \), define the
  endomorphisms \( \alpha_Y,\beta_Y\in \End(\mfa) \) by
  \begin{equation*}
    \alpha_Y(X):=\omega(Y,X)=(\omega_0)^r_{Jr}(Y,X),\qquad \beta_Y(X):=\omega(Y,JX)=(\omega_1)^r_{Jr}(Y,JX)
  \end{equation*}
  for \( X\in \mfa \).  Moreover, we set
  \begin{equation*}
    \nu:=(\omega_1)_{JJ}=(\omega_1)_{JJ}^r\in \Lambda^2 U_J^*\otimes \mfa.
  \end{equation*}
  Moreover, we also use some notation introduced previously without
  further notification.
\end{notation}
We observe the following:

\begin{proposition}
  \label{pro:twostepSKTsheardatatotallyreal}
  Let \( (\mfa,\omega) \) be pre-shear data with \( \mfa \) being
  totally real.
  Then \( (\mfa,\omega) \) is two-step SKT shear data if and only if
  \( f\in S^2 \mfa^*\otimes \mfa \), \( (\omega_1)_{rr}^r=0 \),
  \( \beta_Y=\alpha_{JY} \) for all \( Y\in U_J \), \( \nu \) is of
  type \( (1,1) \) and
  \begin{equation}
    \label{eq:SKTconditionomegaaJzero}
    \cA\Bigl(\omega(\omega(\any,\any),\any)\Bigr)=0,\qquad
    \cA\Bigl(g(\omega(\any,\any),\omega(\any,\any))\Bigr)=0.
  \end{equation}
\end{proposition}

\begin{proof}
  All the conditions but \eqref{eq:SKTconditionomegaaJzero} follow
  from Lemma~\ref{le:Jintegrable}.

  Next, note that \( J^*\omega \) and \( \omega \) both have values in
  \( \mfa=\mfa_r \) while \( J\circ J.\omega \) has values in
  \( J\mfa=J\mfa_r=U_r \).
  Hence, \eqref{eq:AbelianSKTshear1} is equivalent
  \( J^*\omega=\omega \), since this equation says that \( \omega \)
  is of type \( (1,1) \) and so \( J.\omega=0 \).
  Thus,
  \( \omega(J\omega(\any,\any),J\any)=\omega(\omega(\any,\any),\any)
  \), i.e.\
  \( \nu_2(\omega)=\cA\Bigl(\omega(\omega(\any,\any),\any)\Bigr)=0 \),
  and
  \(
  \nu_1(\omega)=\cA\Bigl(g(\omega(\any,\any),\omega(\any,\any))\Bigr)
  \).
  Hence, \eqref{eq:Abeliansheardata} and~\eqref{eq:AbelianSKTshear2}
  are satisfied if and only if
  \begin{equation*}
    \cA\Bigl(\omega(\omega(\any,\any),\any)\Bigr)=0,\qquad \cA\Bigl(g(\omega(\any,\any),\omega(\any,\any))\Bigr)=0.
  \end{equation*}
\end{proof}
\begin{remark}
  Note that \( J^*\omega=\omega \) implies that the induced almost
  complex structure on the shear Lie algebra is Abelian.
\end{remark}
We first look at the exact form of \( f \).
For this, observe that \( f\in S^2 \mfa^*\otimes \mfa \) and
\( f(f(\any,\any),\any)\in S^3\mfa^*\otimes V \).
Moreover, note that the second equation
in~\eqref{eq:SKTconditionomegaaJzero} on
\( U_r^*\otimes U_r^*\otimes \mfa^*\otimes \mfa^* \) is equivalent to
\( g(f(\any,\any),f(\any,\any))\in (\mfa^*)^{\otimes 4} \) being
totally symmetric.

Such an \( f \) can always be brought into a certain ``normal form'':

\begin{lemma}
  \label{le:formoff}
  Let \( (V,g) \) be an \( m \)-dimensional Euclidean vector space and
  \( f\in S^2 V^*\otimes V \) such that both
  \( f(f(\any,\any),\any)\in S^2V^*\otimes V^*\otimes V \) and
  \( g(f(\any,\any),f(\any,\any))\in S^2 V^*\otimes S^2 V^* \) are
  totally symmetric in their \( V^* \)-factors.
  Set \( V_1:=\im(f|_{S^2\im(f)}) \),
  \( V_2:=V_1^{\perp}\cap \im(f) \) and \( r:=\dim(V_1) \).
  Then there exist an orthonormal basis \( (X_1,\dots,X_r) \) of
  \( V_1 \), non-zero
  \( \lambda_1,\ldots,\lambda_r\in \bR\setminus \{0\} \) and a
  complementary subspace \( V_3 \) of \( V_1\oplus V_2 \) in \( V \)
  which is orthogonal to \( V_2 \) satisfying the following
  conditions:
  \begin{enumerate}[(i)]
  \item\label{it:formoff-ev}
    \( f(X_a,X_b) = \delta_{ab}\lambda_{a}X_{a} \) for all
    \( a,b\in\{1,\dots,r\} \),
  \item\label{it:formoff-zero}
    \( f(V_{2}\oplus V_{3},V_{1}\oplus V_{2})=0 \)
  \item\label{it:formoff-V3} \( f(V_{3},V_{3})=V_2 \).
  \end{enumerate}
\end{lemma}

\begin{proof}
  Let \( W=\im(f) \) be the image of~\( f \).
  Then our assumptions imply
  \begin{equation*}
    g(f(\any,\any),\any)|_{W^{\otimes 3}} \in S^{3}W^{*},
  \end{equation*}
  i.e.\ is totally symmetric, since
  \( g(f(\any,f(\any,\any)),f(\any,\any)) \in S^{5}(V^*) \).
  Hence, \( f_X := f(X,\any)|_{W} \in \End(W) \) is self-adjoint with
  respect to \( g \) for all \( X\in \tilde{W} \).
  Now the symmetry assumption implies that all the endomorphisms
  \( f_X \) commute.
  It follows that \( W \) has a common orthonormal basis
  \( X_{1},\dots,X_{m} \) of eigenvectors for these endomorphisms.
  Label these vectors so that \( f(X_{i},X_{i}) \ne 0 \) for
  \( i=1,\dots,r \).
  As
  \( f_{X_{a}}(X_{b}) = f(X_{a},X_{b}) = f_{X_{b}}(X_{a}) \in
  \spa{X_{a}}\cap\spa{X_{b}} \) for all \( a,b\in\{1,\dots,m\} \), we
  find that \( f(X_{a},X_{b}) \) is only non-zero for
  \( a=b \leq r \).
  Thus we get \( W = V_{1} \oplus V_{2} \) orthogonally, with
  \( V_{1}= \im(f|_{S^{2}W}) = \spa{X_{1},\dots,X_{r}} \)
  and~\( V_{2} = \spa{X_{r+1},\dots,X_{m}} \),
  giving~\ref{it:formoff-ev} and \( f(V_{2},V_{1}\oplus V_{2}) = 0 \)
  in~\ref{it:formoff-zero}.

  Note that for \( Y\in V_{2} \) and \( Z\in V \), we have
  \begin{equation*}
    g(f(Y,Z),f(Y,Z)) = g(f(Y,Y),f(Z,Z))=0,
  \end{equation*}
  which implies that \( f(Y,Z)=f(Z,Y)=0 \) for any \( Z\in V \).
  Next, let \( H\colon W^{\bot} \to V \),
  \( H(Y):=Y-\sum_{i=1}^r \frac{1}{\lambda_i}f(Y,X_i) \) and note that
  \( H \) is an injective linear map with \( f(H(Y),X_i)=0 \) for all
  \( i=1\ldots,r \).
  Hence, \( V_3:=H(W^{\bot}) \) is a complement of~\( W \) in~\( V \)
  with \( f(V_3,W)=0 \), and so~\ref{it:formoff-zero} holds.
  Note that \( V_3 \) is not necessarily orthogonal to \( V_1 \) but
  it is orthogonal to~\( V_2 \).  However, we get
  \begin{equation*}
    g(X_i,f(Y,Z))=\frac{1}{\lambda_i}
    g(f(X_i,X_i),f(Y,Z))=\frac{1}{\lambda_i} g(f(X_i,Y),f(X_i,Z))=0
  \end{equation*}
  for all \( i=1,\ldots, r \) and all \( Y,Z\in V_3 \), which implies
  that \( f(Y,Z)\in V_1^{\perp}\cap (\im(f))=V_2 \), as claimed
  in~\ref{it:formoff-V3}.
\end{proof}
In our case, i.e.\ the case of two-step shear data \( (\mfa,\omega) \)
with totally real \( \mfa \), we get \( \mfa_2=\{0\} \) in
Lemma~\ref{le:formoff} and even more the following is true:

\begin{proposition}
  \label{pro:twostepSKTtotallyreal}
  Let \( (\mfa,\omega) \) be two-step SKT shear data with totally real
  \( \mfa \).
  Then, with the notation from the beginning of this subsection, we
  have \( \mfa=\mfa_1\oplus \mfa_3 \), an orthonormal basis
  \( X_1,\ldots,X_r \) of \( \mfa_1 \), \( r:=\dim(\mfa_1) \),
  \( \lambda_1,\ldots,\lambda_r\in \bR\setminus \{0\} \) and one-forms
  \( \mu_1,\ldots,\mu_r \) such that
  \begin{equation}
    \label{eq:twostepSKTtotallyreal1}
    f(X_i,X_j)=\delta_{ij} \lambda_i X_i,\quad f(\mfa_3,\mfa)=0,\quad \alpha_Y(X_i)=\mu_i(Y) X_i,\quad \alpha_Y(\mfa_3)=0
  \end{equation}
  for all \( i,j\in \{1,\ldots,r\} \).
\end{proposition}
\begin{proof}
  Let \( \mfa_1 \), \( \mfa_2 \), \( \mfa_3 \), \( r \), \( s \),
  \( (X_1,\ldots,X_r) \),
  \( \lambda_1,\ldots,\lambda_r\in \bR\setminus \{0\} \) be as in
  Lemma~\ref{le:formoff} and let us first look at
  \( \alpha\in U_J^*\otimes \End(\mfa) \).

  Then the first condition in~\eqref{eq:SKTconditionomegaaJzero} gives
  us that \( \alpha_Y\in \End(\mfa) \) commutes with
  \( f_X:=f(X,\any) \) for any \( X\in \mfa \).
  In particular, \( \alpha_Y \) preserves the eigenspaces of all
  \( f_{X_i} \) and so \( \spa{X_i} \) for all \( i=1,\ldots,r \) as
  well as \( \mfa_2\oplus \mfa_3 \), which is the common zero
  eigenspace of all \( f_{X_i} \), \( i=1,\ldots,r \).
  In particular, there are \( \mu_1,\ldots,\mu_r\in U_J^* \) with
  \( \alpha_Y(X_i)=\mu_i(Y) X_i \) for all \( i=1,\ldots,r \).

  Next, look at the second equation
  in~\eqref{eq:SKTconditionomegaaJzero} and insert into that equation
  \( (X,JX,\allowbreak Y,JY) \) with \( X\in \mfa_2 \) As
  \( f(X,X)=0 \), one gets
  \begin{equation*}
    \begin{split}
      0&=\cA(g(\omega(\any,\any),\omega(\any,\any)))(X,JX,Y,JY)\\
       &=-g(\omega(X,Y),\omega(JX,JY))+g(\omega(X,JY),\omega(JX,Y))\\
       &=-g(\alpha_Y(X),\beta_{JY}(X))+g(\alpha_{JY}(X),\beta_Y(JX))=-\norm{\alpha_Y(X)}^2-\norm{\alpha_{JY}(X)}^2,
    \end{split}
  \end{equation*}
  i.e.\ \( \alpha_Y(X)=0 \) and so \( \alpha_Y(\mfa_2)=0 \), where
  here and in the following we leave out irrelevant normalization
  factors coming from anti-symmetrization map \( \cA \).

  Now take \( X \), \( W\in \mfa_3 \).
  Then \( 0=[f_X,\alpha_Y](W)=f_X(\alpha_Y(W)) \) since
  \( f_X(W)\in \mfa_2\subseteq \ker(\alpha_Y) \).
  Hence, if we denote by \( \mfa_3' \) the common kernel of all
  \( f_X|_{\mfa_3} \), \( X\in \mfa_3 \), then
  \( \alpha_Y(\mfa_3)\subseteq \mfa_2\oplus \mfa_3' \).
  We are going to show that \( \mfa_3'=\mfa_3 \), which implies
  \( f_X=0 \) for all \( X\in \mfa_3 \) and so \( \mfa_2=0 \) by
  condition~\ref{it:formoff-V3} in Lemma~\ref{le:formoff}.

  For this, let \( X\in \mfa_3' \) and \( Y\in U_J \).
  Again, we have \( f(X,X) \) and so
  \begin{equation*}
    \begin{split}
      0&=\cA(g(\omega(\any,\any),\omega(\any,\any)))(X_,JX_,Y,JY)\\
       &=-g(\alpha_Y(X),\beta_{JY}(X))+g(\alpha_{JY}(X),\beta_Y(JX))=-\norm{\alpha_Y(X)}^2-\norm{\alpha_{JY}(X)}^2.,
    \end{split}
  \end{equation*}
  i.e.\ \( \alpha_Y(\mfa_3')=\{0\} \), and so
  \( \alpha_Y(\mfa_2\oplus \mfa_3')=\{0\} \).

  Next let \( X\in \mfa_3 \), \( Y,Z\in U_J \) be given.
  Then~\eqref{eq:SKTconditionomegaaJzero} gives us
  \begin{equation*}
    \begin{split}
      0&=\cA(\omega(\omega(\any,\any),\any))(JX,Y,Z)=f(X,\nu(Y,Z))-\alpha_Y(\alpha_{JZ}(X))+\alpha_{Z}(\alpha_{JY}(X))\\
       &=f(X,\nu(Y,Z))
    \end{split}
  \end{equation*}
  since \( \alpha_{JZ}(X),\alpha_{JY}(X)\in \mfa_2\oplus\mfa_3' \).

  But this implies \( \nu(Y,Z)\in \mfa_1\oplus \mfa_2\oplus\mfa_3' \)
  for any \( Y,Z\in U_J \).
  As the images of \( f \), \( \alpha_Y \) and \( \beta_Z \) are all
  \( \mfa_1\oplus \mfa_2\oplus \mfa_3' \) for any \( Y,Z\in U_J \), we
  have \( \im(\omega)\subseteq \mfa_1\oplus \mfa_2\oplus \mfa_3' \).
  However, by assumption,
  \( \im(\omega)=\mfa=\mfa_1\oplus \mfa_2\oplus \mfa_3 \), and so
  \( \mfa_3'=\mfa_3 \).
  As said above, this implies \( \mfa_2=\{0\} \).
  Note that now also \( 0=\alpha_Y(\mfa_3')=\alpha_Y(\mfa_3) \) for
  any \( Y\in U_J \), which finishes the proof.
\end{proof}
We write now \( \nu=\sum_{i=1}^r \nu_i X_i+\tilde{\nu} \) with
\( \nu_1,\ldots \nu_r\in \Lambda^2 U_J^* \) and
\( \tilde{\nu}\in \Lambda^2 U_J^*\otimes \mfa_3 \).Then, applying the
first part of~\eqref{eq:SKTconditionomegaaJzero} to \( (JX_i,Y,Z) \)
with \( i\in \{1,\ldots,r\} \) and \( Y,Z\in U_J \), one obtains
\begin{equation*}
  \begin{split}
    0&=\cA(\omega(\omega(\any,\any),\any))(JX_i,Y,Z)
       =f(X_i,\nu(Y,Z))-(\mu_i\wedge \mu_i\circ J)(Y,Z)X_i\\
     &=\bigl(\lambda_i \nu_i-(\mu_i\wedge \mu_i\circ J)\bigr)(Y,Z) X_i,
  \end{split}
\end{equation*}
so \( \nu_i=(\mu_i\wedge \mu_i\circ J)/\lambda_i \).

Next, we show that the splitting \( \mfa=\mfa_1\oplus \mfa_3 \) is
\( g \)-orthogonal.
For this, we insert \( (JX_i,X_i,Y,Z) \) with
\( i\in \{1,\ldots,r\} \) and \( Y,Z\in U_J \) into the second
equation in~\eqref{eq:SKTconditionomegaaJzero} and get
\begin{equation*}
  \begin{split}
    0&=\cA\Bigl(g(\omega(\any,\any),\omega(\any,\any))\Bigr)(JX_i,X_i,Y,Z)\\
     &=g(f(X_i,X_i),\nu(Y,Z))+g(\alpha_{JY}(X_i),\alpha_Z(X_i))-g(\alpha_{JZ}(X_i),\alpha_Y(X_i))\\
     &=(\mu_i\wedge \mu_i\circ J)(Y,Z)+\lambda_i g(X_i,\tilde{\nu}(Y,Z))-(\mu_i\wedge \mu_i\circ J)(Y,Z)\\
     &=\lambda_i g(X_i,\tilde{\nu}(Y,Z)),
  \end{split}
\end{equation*}
thus \( \tilde{\nu}(Y,Z)\perp_g \mfa_1 \) for all \( Y,Z\in U_J \).
As by assumption \( \im(\omega)=\mfa \) and so \( \mfa_3 \) has to be
spanned by \( \im(\tilde{\nu}) \), we obtain
\( \mfa_1\perp_g \mfa_3 \).
Now one checks that~\eqref{eq:SKTconditionomegaaJzero} is valid with
the exception of the second equation on \( \Lambda^4 U_J^* \), which
is equivalent to
\( \sum_{i=r+1}^m \tilde{\nu}_i\wedge \tilde{\nu}_i=0 \) if we write
\( \tilde{\nu}=\sum_{i=r+1}^m \tilde{\nu}_i X_i \) for an orthonormal
basis \( X_{r+1},\dots,X_m \) of \( \mfa_3 \) with
\( \tilde{\nu}_{r+1},\ldots,\tilde{\nu}_m\in \Lambda^2 U_J^* \).
Summarizing, we have obtained:

\begin{theorem}
  \label{th:totallyrealgeneralresult}
  Let \( (\mfg,g, J) \) be a \( 2n \)-dimensional almost Hermitian
  two-step solvable Lie algebra with totally real \( m \)-dimensional
  commutator ideal \( \derg \).

  Then \( (\mfg,g,J) \) is SKT if and only if there exists an
  orthonormal basis \( (X_1,\ldots,X_m) \) of \( \derg \),
  \( \lambda_1,\ldots,\lambda_r\in \bR\setminus \{0\} \),
  \( \mu_1,\ldots,\mu_r\in U_J^* \) with
  \( U_J:=(\derg\oplus J\derg)^{\perp} \) and
  \( \tilde{\nu}_{r+1},\ldots,\tilde{\nu}_{n}\in \Lambda^2 U_J^* \)
  such that the only non-zero Lie brackets (up to anti-symmetry) are
  given by
  \begin{equation*}
    \begin{split}
      [JX_i,X_i]&=\lambda_i X_i,\ [Y,X_i]=\mu_i(Y) X_i,\
                  [Y,JX_i]=-\mu_i(JY) X_i,\\
      [Y,Z]&=\sum_{i=1}^r \frac{1}{\lambda_i}(\mu_i\wedge \mu_i\circ
             J)(Y,Z) X_i + \sum_{j=r+1}^m \tilde{\nu}_j(Y,Z) X_j  ,
    \end{split}
  \end{equation*}
  for \( i=1,\ldots,r \), \( Y,Z\in U_J \), and such that for
  \( \tilde{\nu}:=\sum_{i=r+1}^m \tilde{\nu}_i X_i \) we have
  \( \sum_{i=r+1}^m \tilde{\nu}_i\wedge \tilde{\nu}_i=0 \) and
  \( \im(\tilde{\nu})=\spa{X_{r+1},\ldots,X_m} \).

  If this is the case, we have
  \( \mfg\cong r\,\aff_{\bR}\oplus \mfh \) for some
  \( 2(n-r) \)-dimensional nilpotent Lie algebra \( \mfh \).
\end{theorem}

If \( m=r \) in Theorem~\ref{th:totallyrealgeneralresult}, then
\( \tilde{\nu} \) does not occur anymore and we have the following
result.

\begin{corollary}
  \label{co:totallyreal1}
  Let \( (\mfg,g, J) \) be a \( 2n \)-dimensional two-step solvable
  almost Hermitian Lie algebra with totally real \( m \)-dimensional
  commutator ideal \( \derg \) such that \( [\derg,J\derg]=\derg \).

  Then \( (\mfg,g,J) \) is SKT if and only if there exists an
  orthonormal basis \( (X_1,\ldots,X_m) \) of \( \derg \),
  \( \lambda_1,\ldots,\lambda_m\in \bR\setminus \{0\} \),
  \( \mu_1,\ldots,\mu_m\in U_J^* \) with
  \( U_J:=(\derg\oplus J\derg)^{\perp} \), such that the only non-zero
  Lie brackets (up to anti-symmetry) are given by
  \begin{equation*}
    \begin{split}
      [JX_i,X_i]
      &=\lambda_i X_i,\quad
        [Y,X_i]=\mu_i(Y) X_i,\quad
        [Y,JX_i]=-\mu_i(JY) X_i, \\
      [Y,Z]&=\sum_{i=1}^m \frac{1}{\lambda_i}(\mu_i\wedge \mu_i\circ J)(Y,Z)
             X_i,
    \end{split}
  \end{equation*}
  for \( i=1,\ldots,m \), \( Y,Z\in U_J \), If this is the case, we
  have \( \mfg\cong m \, \aff_{\bR} \oplus \bR^{2n-2m} \).
\end{corollary}

For \( n=m \) in Corollary~\ref{co:totallyreal1} we get:

\begin{corollary}
  \label{co:totallyreal2}
  A \( 2n \)-dimensional almost Hermitian two-step solvable Lie
  algebra \( (\mfg,g, J) \) with totally real commutator ideal
  \( \derg \) of dimension \( n \) is SKT if and only if there exists
  an orthonormal basis \( (X_1,\ldots,X_n) \) of \( \derg \) and
  \( \lambda_1,\ldots,\lambda_n\in \bR\setminus \{0\} \) such that the
  only non-zero Lie brackets (up to anti-symmetry) are given by
  \( [JX_i,X_i]=\lambda_i X_i \) for \( i=1,\ldots,n \)

  Moreover, we then have \( \mfg\cong n \,\aff_{\bR} \).
\end{corollary}

For \( r<m \), we still need to solve
\( \sum_{i=r+1}^m \tilde{\nu}_i\wedge \tilde{\nu}_i=0 \) and
\( \im(\tilde{\nu})=\spa{X_{r+1},\ldots,\allowbreak X_m} \) in
Theorem~\ref{th:totallyrealgeneralresult}.
This is not too difficult for \( r\in \{m-2,m-1\} \), and so we will
do that in the following.
For smaller~\( r \), solving these equations seems to be more involved,
cf.\ Remark~\ref{re:highercodim}.

Let us start with \( r=m-1 \).
Then \( \tilde{\nu}_m\wedge \tilde{\nu}_m=0 \),
\( \im(\tilde{\nu})=\spa{X_m} \) and \( \tilde{\nu}_m \) being a
\( (1,1) \)-form implies \( \tilde{\nu}_m=\alpha\wedge J^*\alpha \)
for some \( \alpha\in U_J^*\setminus \{0\} \).

Next, consider the case \( r=m-2 \).
Then we need to solve the equation
\( \tilde{\nu}_{m-1}^2+\tilde{\nu}_m^2=0 \), which is done by the next
lemma.

\begin{lemma}
  \label{le:quadraticequationzero}
  Let \( V \) be a \( 2n \)-dimensional vector space with complex
  structure \( J \) and let real
  \( \nu_1,\nu_2\in \Lambda^{1,1} V^* \) be given with
  \( \nu_1^2+\nu_2^2=0 \).
  Then, up to a rotation in \( \spa{\nu_1,\nu_2} \), there are
  \( \alpha,\beta\in V^* \) such that
  \( \nu_1=\alpha\wedge J^*\alpha \) and
  \( \nu_2=\beta\wedge J^*\beta \).
\end{lemma}

\begin{proof}
  First of all, we show that if \( \nu_1^2+\nu_2^2=0 \) implies then
  both \( \nu_1 \) and \( \nu_2 \) has rank at most four.
  For this, assume contrary, without loss of generality, that
  \( \nu_1 \) has rank at least six.
  Then we may find a decomposition \( V=V_1\oplus V_2 \) with
  \( \dim(V_1)=4 \) such that \( \nu_1=\omega_1+\tau_1 \) with
  \( \omega_1\in\Lambda^2 V_1^* \) being of rank four and
  \( 0\neq\tau_1\in \Lambda^2 V_2^* \).
  We may decompose \( \tilde{\mu}_2=\omega_2+\rho_2+\tau_2 \) with
  \( \omega_2\in\Lambda^2 V_1^* \), \( \rho_2\in V_1^*\wedge V_2^* \)
  and \( \tau_2\in \Lambda^2 V_2^* \).
  Looking at the \( \Lambda^4 V_1^* \)-component and the
  \( \Lambda^3 V_1^*\wedge V_2^* \)-component of
  \( \nu_1^2+\nu_2^2=0 \), we obtain
  \begin{equation*}
    \omega_1^2+\omega_2^2=0,\quad \omega_2\wedge \rho_2=0
  \end{equation*}
  Due to \( \omega_1^2\neq 0 \), we must have \( \omega_2\neq 0 \) and
  so the second equation implies \( \rho_2=0 \).
  But then the \( \Lambda^2 V_1^*\wedge \Lambda^2 V_2^* \)-component
  of \( \nu_1^2+\nu_2^2=0 \) gives us
  \begin{equation*}
    \omega_1\wedge \tau_1+\omega_2\wedge \tau_2=0.
  \end{equation*}
  As \( \omega_2^2=-\omega_1^2 \), the two-forms \( \omega_1 \),
  \( \omega_2 \) cannot be linearly dependent and so we must have
  \( \tau_1=\tau_2=0 \), a contradiction.
  Hence, both \( \nu_1 \) and \( \nu_2 \) have rank at most four.
  If they both have rank four, then they must have the same kernel
  since if \( X \) is in the kernel of say \( \nu_1 \), then
  \( 0=X\hook (\nu_1^2+\nu_2^2)=2(X\hook \nu_2)\wedge \nu_2 \), which
  forces \( X\hook \nu_2=0 \) as well.
  Moreover, it cannot be that exactly one of the two-form
  \( \omega_2 \), \( \omega_1 \) has rank four so the only other
  option is that both have rank two.
  So in this case, we have \( \nu_1=\alpha\wedge J^*\alpha \) and
  \( \nu_2=\beta\wedge J^*\beta \) for \( \alpha,\beta\in U_J^* \).

  Let us now come back to the case that both \( \nu_1 \) and
  \( \nu_2 \) have rank four, and so have common support.
  Setting \( \Omega:=\nu_1^2 \), we have \( \nu_2^2=-\Omega \) and
  \( \nu_1\wedge \nu_2=a\Omega \) for some \( a\in \bR \).
  As the matrix
  \( \begin{psmallmatrix} 1 & a \\ a & -1\end{psmallmatrix} \) is
  symmetric, we may apply an orthogonal transformation to
  \( \spa{X_1,X_2} \) such that with respect to the new basis we get
  \( a=0 \).
  But then
  \( (\frac{1}{\sqrt{2}}\nu_1\pm \frac{1}{\sqrt{2}}\nu_2)^2=0 \) and
  so by applying another orthogonal transformation to
  \( \spa{\nu_1,\nu_2} \), we may also here assume that both
  \( \nu_1 \) and \( \nu_2 \) have rank two.
  Hence, we always can rotate the pair \( (\nu_{1},\nu_{2}) \) to get
  \( \nu_1=\alpha\wedge J^*\alpha \) and
  \( \nu_2=\beta\wedge J^*\beta \) for certain
  \( \alpha,\beta\in V^* \)
\end{proof}

Coming back to our case, we thus have
\( \tilde{\nu}_{m-1}=\alpha\wedge J^*\alpha \) and
\( \tilde{\nu}_m=\beta\wedge J^*\beta \) for certain
\( \alpha,\beta\in U_J^* \).
Additionally, we know that we must have
\( \im(\nu)=\spa{X_{m-1},X_m} \).
But so \( \tilde{\nu}_{m-1},\tilde{\nu}_m \) have to be
\( \bR \)-linearly independent, which implies that \( \alpha,\beta \)
are \( \bC \)-linearly independent.  Hence, we have obtained:

\begin{theorem}
  \label{th:totallyreal2}
  Let \( (\mfg,g, J) \) be a \( 2n \)-dimensional almost Hermitian
  two-step solvable Lie algebra with totally real \( m \)-dimensional
  commutator ideal \( \derg \) such that \( [\derg,J\derg] \) has
  codimension \( \ell \in \{1,2\} \) in \( \derg \).
  Moreover, set \( U_J:=(\derg\oplus J\derg)^{\perp} \).
  \begin{enumerate}[(a)]
  \item If \( \ell =1 \), then \( (\mfg,g,J) \) is SKT if and only if
    there exists an orthonormal basis \( (X_1,\ldots,X_m) \) of
    \( \derg \),
    \( \lambda_1,\ldots,\lambda_{m-1}\in \bR\setminus \{0\} \),
    \( \mu_1,\ldots,\mu_{m-1},\alpha \in U_J^* \) with
    \( \alpha\neq 0 \) such that the only non-zero Lie brackets (up to
    anti-symmetry) are given by
    \begin{equation*}
      \begin{split}
        [JX_i,X_i]&=\lambda_i X_i,\ [Y,X_i]=\mu_i(Y) X_i,\ [Y,JX_i]=-\mu_i(JY)
                    X_i,\\
        [Y,Z]&=\sum_{i=1}^{m-1} \frac{1}{\lambda_i}(\mu_i\wedge
               \mu_i\circ J)(Y,Z) X_i+(\alpha\wedge
               J^*\alpha)(Y,Z)X_m,
      \end{split}
    \end{equation*}
    for \( i=1,\ldots,m-1 \), \( Y,Z\in U_J \).
    In this case, we have
    \( \mfg\cong {(m-1)}\, \aff_{\bR}\oplus \mfh_3 \oplus
    \bR^{2n-2m-1} \).
  \item If \( \ell =2 \), then \( (\mfg,g,J) \) is SKT if and only if
    there exists an orthonormal basis \( (X_1,\ldots,X_m) \) of
    \( \derg \),
    \( \lambda_1,\ldots,\lambda_{m-2}\in \bR\setminus \{0\} \),
    \( \mu_1,\ldots,\mu_{m-2},\alpha_1,\alpha_2 \in U_J^* \) with
    \( \alpha_1 \), \( \alpha_2 \) being \( \bC \)-linearly
    independent such that the only non-zero Lie brackets (up to
    anti-symmetry) are given by
    \begin{equation*}
      \begin{split}
        [JX_i,X_i]&=\lambda_i X_i,\ [Y,X_i]=\mu_i(Y) X_i,\
                    [Y,X_i]=-\mu_i(JY) X_i,\\
        [Y,Z]&=\sum_{i=1}^{m-2} \frac{1}{\lambda_i}(\mu_i\wedge
               \mu_i\circ J)(Y,Z) X_i+\sum_{i=1}^2(\alpha_i\wedge
               J^*\alpha_i)(Y,Z)X_{m-2+i},
      \end{split}
    \end{equation*}
    for \( i=1,\ldots,m-2 \), \( Y,Z\in U_J \).
    In this case, we have
    \( \mfg\cong (m-2)\, \aff_{\bR}\oplus 2\mfh_3 \oplus \bR^{2n-2m-2}
    \).
  \end{enumerate}
\end{theorem}

\begin{remark}
  \label{re:highercodim}
  Note that for higher codimensions \( \ell \) it is not clear whether
  Lemma~\ref{le:quadraticequationzero} generalizes in the natural way
  to more than two \( (1,1) \)-forms \( \nu_i \).
  So one is left with the description from
  Theorem~\ref{th:totallyrealgeneralresult} for these cases.
  Note further that even if all \( \nu_i \) are decomposable, still
  the nilpotent Lie algebra \( \mfh \) from
  Theorem~\ref{th:totallyrealgeneralresult} is, in general, not equal
  to a sum of the form
  \( r\aff_{\bR}\oplus \ell \mfh_3\oplus \bR^{2(n-r)-3s} \) as in the
  case of \( \ell \in \{1,2\} \).

  For example, if \( \ell =3 \), take \( \bC \)-linearly independent
  \( \alpha_1,\alpha_2\in U_J \) and set
  \( \nu_{m-2}:=\alpha_1\wedge J^*\alpha_1 \),
  \( \nu_{m-1}:=\alpha_2\wedge J^*\alpha_2 \),
  \( \nu_m:=(\alpha_1+\alpha_2)\wedge J^*(\alpha_1+\alpha_2) \).
  Then \( \mfh\cong r\aff_{\bR}\oplus (37D)\oplus \bR^{2(n-r)-7} \)
  with the seven-dimensional nilpotent Lie algebra \( 37D \)
  from~\cite{Gong}.
\end{remark}

Now all Lie algebras up to dimension two are solvable and the only
nilpotent Lie algebras up to this dimension are the Abelian ones.
As the commutator ideal \( \derg \) of a solvable Lie algebra is
nilpotent, one sees that any Lie algebra with one-dimensional or
two-dimensional commutator ideal is two-step solvable.
Now surely if the dimension of \( \derg \) is one, \( \derg \) is
totally real and the codimension of \( [\derg,J\derg] \) in
\( \derg \) is at most one.
Similarly, if the dimension of \( \derg \) is two and \( \derg \) is
totally real, then the codimension of \( [\derg,J\derg] \) in
\( \derg \) is at most two.
Hence, Corollary~\ref{co:totallyreal1} and
Theorem~\ref{th:totallyreal2} imply

\begin{corollary}
  \label{co:lowdimnotJinv}
  Let \( (\mfg,g, J) \) be a \( 2n \)-dimensional almost Hermitian Lie
  algebra.
  \begin{enumerate}[(a)]
  \item If \( \dim(\derg)=1 \), then \( (\mfg,g,J) \) is an SKT Lie
    algebra if and only if there exists \( 0\neq X\in \derg \) of norm
    one such that
    \begin{enumerate}[(i)]
    \item there exists \( \lambda\in \bR\setminus \{0\} \) and
      \( \mu\in U_J^* \) so that the only non-zero Lie brackets (up to
      anti-symmetry) are given by
      \begin{equation*}
        \begin{split}
          [JX,X]&=\lambda X,\ [Y,X]=\mu(Y) X,\ [Y,JX]=-\mu(JY) X
                  \eqcond{for \( Y\in U_J \),}\\
          [Y,Z]&=\frac{1}{\lambda}(\mu\wedge \mu\circ J)(Y,Z) X
                 \eqcond{for \( Y,Z\in U_J \)}
        \end{split}
      \end{equation*}
      or
    \item there exists \( \alpha\in U_J^*\setminus \{0\} \) so that
      the only non-zero Lie bracket is given by
      \begin{equation*} [Y,Z]=(\alpha\wedge J^*\alpha)(Y,Z) X
      \end{equation*}
      for \( Y,Z\in U_J \).
    \end{enumerate}
    In the first case, we have
    \( \mfg\cong \aff_{\bR}\oplus \bR^{2n-2} \) whereas in the second
    case we have \( \mfg\cong \mfh_3\oplus \bR^{2n-3} \).
  \item If \( \dim(\derg)=2 \) and \( \derg \) is totally real, then
    \( (\mfg,g,J) \) is an SKT Lie algebra if and only if there exists
    an orthonormal basis \( (X_1,X_2) \) of \( \derg \) such that
    \begin{enumerate}[(i)]
    \item there exist \( \lambda_1,\lambda_2\in \bR\setminus \{0\} \)
      and \( \mu_1,\mu_2\in U_J^* \) so that the only non-zero Lie
      brackets (up to anti-symmetry) are given by
      \begin{equation*}
        \begin{split}
          [JX_i,X_i]&=\lambda X,\ [Y,X_i]=\mu(Y) X_i,\ [Y,JX_i]=-\mu(JY) X_i
                      \eqcond[1]{for \( i=1,2 \), \( Y\in U_J \),}\\
          [Y,Z]&=\sum_{i=1}^2\frac{1}{\lambda_i}(\mu_i\wedge \mu_i\circ J)(Y,Z)
                 X_i \eqcond{for \( Y,Z\in U_J \),}
        \end{split}
      \end{equation*}
    \item there exist \( \lambda\in \bR\setminus \{0\} \) and
      \( \mu,\alpha\in U_J^* \), \( \alpha\neq 0 \) so that the only
      non-zero Lie brackets (up to anti-symmetry) are given by
      \begin{equation*}
        \begin{split}
          [JX_1,X_1]&=\lambda X_1,\ [Y,X_1]=\mu(Y) X_1,\
                      [Y,JX_1]=-\mu(JY) X_1
                      \eqcond{for \( Y\in U_J \),}\\
          [Y,Z]&=\frac{1}{\lambda}(\mu\wedge \mu\circ J)(Y,Z) X_1
                 + (\alpha\wedge J^*\alpha)(Y,Z)X_2
                 \eqcond{for \( Y,Z\in U_J \)}
        \end{split}
      \end{equation*}
      or
    \item there exist \( \bC \)-linearly independent
      \( \alpha,\beta\in U_J^* \) so that the only non-zero Lie
      bracket is given by
      \begin{equation*} [Y,Z]=(\alpha\wedge J^*\alpha)(Y,Z)
        X_1+(\beta\wedge J^*\beta)(Y,Z) X_2
      \end{equation*}
      for \( Y,Z\in U_J \).
    \end{enumerate}
    In the first case, we have
    \( \mfg\cong 2\aff_{\bR}\oplus \bR^{2n-4} \), in the second case,
    we have \( \mfg\cong \aff_{\bR}\oplus \mfh_3\oplus \bR^{2n-5} \)
    and in the third case, we have
    \( \mfg\cong 2\mfh_3\oplus \bR^{2n-6} \).
  \end{enumerate}
\end{corollary}

\begin{remark}
  Note that by~\cite[Proposition 2.2]{ABD}, a \( 2n \)-dimensional Lie
  algebras \( \mfg \) with one-dimensional commutator ideal is either
  isomorphic to \( \mfh_{2k+1}\oplus \bR^{2(n-k)-1} \) or
  \( \aff_{\bR}\oplus \bR^{2n-2} \) for some \( k\in \bR \) and that
  all these Lie algebras admit an Abelian complex structure \( J \).
  So we have shown that on \( \mfh_{2k+1}\oplus \bR^{2(n-k)-1} \) with
  \( k\geq 2 \) any of these Abelian complex structures cannot come
  from an SKT structure on \( \mfg \).
\end{remark}

\begin{corollary}
  \label{co:6dtotallyreal}
  A six-dimensional two-step solvable Lie algebra \( \mfg \) admits an
  SKT structure with totally real \( \derg \) if and only if
  \( \mfg \) is isomorphic to one of the following Lie algebras:
  \begin{equation*}
    \bR^6, \aff_{\bR}\oplus \bR^4, \mfh_3\oplus \bR^3,
    2\aff_{\bR}\oplus \bR, \aff_{\bR}\oplus \mfh_3\oplus \bR,
    \mfh_3\oplus \mfh_3, 3\aff_{\bR}.
  \end{equation*}
  \begin{remark}
    In~\cite{ABD}, all six-dimensional Lie algebras with an Abelian
    complex structure have been classified.
    All these Lie algebra are two-step solvable since this is true in
    general for Lie algebras endowed with an Abelian complex
    structure.
    There are in total \( 8 \) (including the Abelian Lie algebra
    \( \bR^6) \) six-dimensional Lie algebras and \( 12 \) single and
    \( 1 \) two-parameter family of six-dimensional non-nilpotent Lie
    algebras admitting an Abelian complex structure.
    By our results, \( 3 \) out of these \( 8 \) nilpotent and \( 4 \)
    out of the \( 12 \) single non-solvable admits Abelian complex
    structures coming from an SKT structure with totally real
    commutator ideal.
    Note that such an Abelian complex structure does not exist on any
    of the non-nilpotent Lie algebras in the two-parameter family.
  \end{remark}
\end{corollary}

\section{Commutator ideal complex of real-dimension two}
\label{sec:2dcomplexcommutator}

Here, we consider two-step SKT shear data \( (\omega,\mfa) \) with
\( \mfa \) being complex and of dimension two.
Note that we then have \( \mfa=\mfa_J \), \( \mfa_r=0 \),
\( U_r=J\mfa_r=0 \) and so \( U=U_J \).

We first show that \( \omega(Z,\any)|_{\mfa} \) acts on \( \mfa \) by
multiplication with an imaginary number.
The proof will use a useful interpretation of the second part of
Lemma~\ref{le:Jintegrable}\ref{item:o0-JJr} in dimension~\( 2 \).

\begin{lemma}
  \label{lem:2d-comm-Jrel}
  Suppose \( V=\bR^{2} \), with the standard complex
  structure~\( J \).  If \( A_{1},A_{2} \in \bR^{2\times 2} \) satisfy
  \begin{enumerate*}[(a)]
  \item \( [A_{1},A_{2}] = 0 \) and
  \item \( [A_{2},J] = J[A_{1},J] \),
  \end{enumerate*}
  then both \( A_{1} \) and~\( A_{2} \) commute with~\( J \).
\end{lemma}

\begin{proof}
  Decompose \( A_i=A_i^J+A_i^{J-} \) with \( JA_i^J = A_{i}^{J}J \)
  and \( JA_{i}^{J-} = - A_{i}^{J-}J \).
  Then \( \dim_{\bR} V = 2 \), implies
  \( [A_{1}^{J},A_{2}^{J}] = 0 \), so (a)~implies the
  \( J \)-anti-invariant parts \( A_{i}^{J-} \) commute.
  These parts are symmetric matrices, so can be simultaneously
  diagonalised.
  However, (b)~gives
  \( 2 A_2^{J-} J = [A_2,J] = J[A_1,J] = 2 A_1^{J-} \), which for
  diagonal matrices implies \( A_{i}^{J-} = 0 \).
\end{proof}

\begin{lemma}
  \label{le:2dcomplex}
  Let \( (\omega,\mfa) \) be two-step SKT shear data with \( \mfa \)
  complex and of dimension two.
  Then there exists \( \alpha\in U^* \) such that
  \( \omega(Z,X)=\alpha(Z) JX \) for any \( Z\in U \),
  \( X\in \mfa \).
\end{lemma}

\begin{proof}
  Fix \( Z\in U \) of norm one and set
  \( A_1:=\omega(Z,\any)|_{\mfa}\in \End(\mfa) \),
  \( A_2:=\omega(JZ,\any)|_{\mfa}\in \End(\mfa) \).
  Inserting \( (Z,JZ,X_1,X_2) \), for \( X_1,X_2\in \mfa \),
  into~\eqref{eq:AbelianSKTshear2}, we obtain
  \begin{equation}
    \label{eq:shear2-A1A2}
    0 = JA_2^t A_2 + A_2^t A_2 J + JA_1^t A_1 + A_1^t A_1 J + A_1^t J
    A_1^J + A_1 J A_1 + A_2^t J A_2^J + A_2 J A_2.
  \end{equation}
  Equation~\eqref{eq:Abeliansheardata}~implies \( [A_1,A_2]=0 \) and
  Lemma~\ref{le:Jintegrable}\ref{item:o0-JJr} gives
  \( [A_2,J]=J[A_1,J] \), so by Lemma~\ref{lem:2d-comm-Jrel} both
  \( A_{i} \) commute with~\( J \).
  Considering \( \mfa \) as a complex one-dimensional space, we
  identify \( A_i \) with \( z_i\in \bC \).
  As \( J \)~becomes multiplication by~\( i \) and transpose is
  complex conjugation, equation~\eqref{eq:shear2-A1A2} is equivalent
  to
  \begin{equation*}
    0=\abs{z_1}^2+\abs{z_2}^2+z_1^2+z_2^2=\Re(z_1)^2+\Re(z_2)^2,
  \end{equation*}
  so \( \Re(z_1)=0 =\Re(z_2) \), giving the claimed statement.
\end{proof}

\begin{remark}
  Let \( (\mfa,\omega) \) be pre-shear data with \( \mfa \) being
  complex and of dimension two such that there is some
  \( \alpha\in U^* \) with \( \omega(Z,X)=\alpha(Z) JX \) for any
  \( Z\in U \), \( Y\in \mfa \).
  Then~\eqref{eq:Abeliansheardata} is fulfilled on
  \( \mfa\wedge \Lambda^2 U \), \eqref{eq:AbelianSKTshear1}~is
  satisfied on \( \mfa\wedge U \) and \eqref{eq:AbelianSKTshear2}~is
  fulfilled on \( \Lambda^2 \mfa \wedge \Lambda^2 U \).
\end{remark}
Next, we are going to show that if \( \alpha\neq 0 \), then
\( \nu_i=\lambda_i \alpha\wedge J^*\alpha \) for some
\( \lambda_i\in \bR \), \( i=1,2 \), where
\( \nu_1,\nu_2\in \Lambda^2 U^* \) are uniquely defined by
\( \nu=\sum_{i=1}^2 \nu_i X_i \).

First of all, \eqref{eq:Abeliansheardata} is valid on
\( \Lambda^3 U \) if and only if \( \alpha\wedge \nu_1=0 \),
\( \alpha\wedge \nu_2=0 \).
Next, \eqref{eq:AbelianSKTshear2}~on \( \mfa\wedge \Lambda^3 U \)
gives us
\begin{equation*}
  \begin{split}
    0&=J^*\nu_2\wedge \alpha+J^*\alpha\wedge \nu_2+\alpha \wedge \nu_1=J^*\nu_2\wedge \alpha+J^*\alpha\wedge \nu_2,\\
    0&=J^*\nu_1\wedge \alpha+J^*\alpha\wedge \nu_1+\alpha \wedge \nu_1=J^*\nu_1\wedge \alpha+J^*\alpha\wedge \nu_1.
  \end{split}
\end{equation*}
Now we may write \( \nu_i=\alpha\wedge \beta_i \) for some
\( \beta_i\in U^* \), \( i=1,2 \), and then the just obtained
equations give us \( \alpha\wedge J^*\alpha\wedge \beta_i=0 \) for all
\( i=1,2 \).
Hence, if \( \alpha\neq 0 \), then
\( \beta_i\in \spa{\alpha,J^*\alpha} \), i.e.\
\( \nu_i=\lambda_i \alpha\wedge J^*\alpha \) for some
\( \lambda_i\in \bR \), \( i=1,2 \).
In this case, one easily checks that
then~\eqref{eq:Abeliansheardata}--\eqref{eq:AbelianSKTshear2} are
fulfilled and \( \mfg\cong \mfr'_{3,0}\oplus \bR^{2n-3} \).

Otherwise, \( \alpha=0 \) and so \( \mfg \) is nilpotent.
Then~\eqref{eq:AbelianSKTshear1} evaluated on \( \Lambda^2 U \) gives
us \( \nu_1=\tau_1+\Re(\theta) \), \( \nu_2=\tau_2-\Im(\theta) \) for
certain \( \tau_1,\tau_2\in \Lambda^{1,1} U^* \),
\( \theta\in \Lambda^{2,0} U^* \).
Then, we still have to solve~\eqref{eq:AbelianSKTshear2} on
\( \Lambda^4 U^* \), which gives us
\begin{equation*}
  0=J^*\nu_1\wedge \nu_1+J^*\nu_2\wedge \nu_2=(\tau_1-\Re(\theta))\wedge (\tau_1+\Re(\theta))+(\tau_2-\Im(\theta))\wedge (\tau_2+\Im(\theta)),
\end{equation*}
which is equivalent to
\begin{equation*}
  \tau_1\wedge \tau_1+\tau_2\wedge \tau_2=\Re(\theta)\wedge \Re(\theta)+\Im(\theta)\wedge \Im(\theta)=\theta\wedge \overline{\theta}.
\end{equation*}
We will now show the following

\begin{lemma}
  \label{le:quadraticequationnonzero}
  Let \( (V,J) \) be a vector space with a complex structure.
  Suppose that \( \omega \in \Lambda^{2,0}V \) and real
  \( \sigma_i \in \Lambda^{1,1}V \) satisfy
  \begin{equation}
    \label{eq:sq-2}
    \sigma_1^2 + \sigma_2^2 = \omega \wedge \overline\omega.
  \end{equation}
  Then the complex rank of \( \omega \) is at most two and the common
  kernel of \( \sigma_1,\sigma_2,\omega \) has codimension at most
  six.
\end{lemma}

\begin{proof}
  If \( \omega=0 \), the result follows from
  Lemma~\ref{le:quadraticequationzero}.

  Otherwise, there is a decomposition \( V=U\oplus W \) with
  \( \dim(U)=4 \) such that \( \omega=\alpha\wedge \beta+\omega' \)
  with \( \alpha,\beta\in \Lambda^{1,0} U^* \) \( \bC \)-linearly
  independent and \( \omega'\in \Lambda^{2,0} W^* \).
  We then have
  \( \Lambda^2 V^* = \Lambda^2 U^*+ U^*\wedge W^* + \Lambda^2 W^* \)
  and correspondingly write elements of \( \Lambda^2 V^* \) as
  \( \gamma = \gamma^U + \gamma^{m} + \gamma^W \).

  We first show that \( \omega'=0 \), i.e.\ that \( \omega \) has
  complex rank two.
  For this, put \( a = \Re\alpha \), \( b = \Re\beta \).
  We may use the action of \( \mathrm{SL}(2,\bC) \) as
  \( \mathrm{SO}_0(1,3) \) on \( \Lambda^{1,1}(\bC^2)^* \) to get that
  \( \sigma_j=p (a\wedge Ja + b\wedge Jb) \) for some \( p\in \bR \)
  and then may use the remaining \( \mathrm{SU}(2) \)-action on
  \( \Lambda^2_-U^* \) to get
  \begin{equation*}
    \sigma_2^U = q a\wedge Ja + r b\wedge Jb,
  \end{equation*}
  for some \( q,r \in \bR \).
  As
  \( (\sigma_1^U)^2 + (\sigma_2^U)^2 = 4a\wedge Ja \wedge b \wedge Jb
  \), we get \( p^2+q r=0 \).
  Thus \( (p,q)\neq (0,0) \) and we may rotate
  \( (\sigma_1,\sigma_2) \) in such a way that
  \( \sigma_1=p\; a\wedge Ja+q\; b\wedge Jb \) and
  \( \sigma_2=r\; b\wedge Jb \) for certain \( p,q,r\in \bR \) with
  \( p>0 \).

  We may now write
  \( \sigma_j^m = a\wedge x_j + Ja\wedge Jx_j + b\wedge y_j + Jb
  \wedge Jy_j \) for some \( x_j,y_j \in W \).
  As \( \omega\wedge\overline{\omega} \) has no component in
  \( \Lambda^3U^*\wedge W^* \), we get
  \( \sigma_1^U\wedge\sigma_1^m + \sigma_2^U\wedge\sigma_2^m = 0 \).
  The \( a\wedge Ja\wedge b \) term in this equation gives
  \( p y_1 = 0 \), so \( y_1 = 0 \).

  Now the contribution to \( \Lambda^2U^*\wedge\Lambda^2W^* \) is
  \( \sum_{j=1}^2 (2\sigma_j^U\wedge\sigma_j^W + (\sigma_j^m)^2) =
  (a-iJa)\wedge(b-iJb) \wedge \overline{\omega'} + \omega' \wedge
  (a+iJa)\wedge(b+iJb) \).
  Consider the terms multiplying \( a\wedge b \).
  There is no contribution from \( \sigma_j^U \) nor \( \sigma_1^m \),
  so we get \( -x_2\wedge y_2 = \Re(\omega') \).
  Next, consider the terms multiplying \( a\wedge b \).
  In this case, there is again no contribution from \( \sigma_j^U \)
  and \( \sigma_1^m \) and we have \( x_2\wedge Jy_2=\Im(\omega') \).
  So both \( \Re(\omega') \) and \( \Im(\omega') \) have rank two and,
  consequently, \( \omega'\wedge\overline{\omega'} = 0 \).
  Thus \( \omega' = 0 \) and \( \omega \) has rank two.
  Moreover, \( x_2=0 \) or \( y_2=0 \).
  We look again at the equation
  \( \sigma_1^U\wedge\sigma_1^m + \sigma_2^U\wedge\sigma_2^m = 0 \).
  For \( x_2=0 \), the \( b\wedge Jb\wedge a \)-term gives us
  \( q x_1=0 \) and so, since \( q\neq 0 \), that \( x_1=0 \) as well.
  For \( y_2=0 \), the same term gives us \( qx_1+r x_2=0 \), i.e.\
  \( x_2=-\frac{r}{q} x_1 \).

  Next, looking at the \( \Lambda^4 U^* \)-terms, we have
  \( (\sigma_1^W)^2+(\sigma_2^W)^2=0 \) and so
  Lemma~\ref{le:quadraticequationzero} shows the existence of
  \( c,e\in W^* \) such that \( \sigma_1^W=c\wedge Jc \) and
  \( \sigma_2^W=e\wedge Je \).
  Next, look at the \( U^*\wedge \Lambda^3 W^* \)-terms, i.e.\ at the
  equation
  \( \sigma_1^m\wedge \sigma_1^W+\sigma_2^m\wedge \sigma_2^W=0 \).

  If \( x_2=0 \), then we saw above that \( \sigma_1^m=0 \) and so we
  get \( y_1\wedge e\wedge Je=0 \), i.e.\ \( y_1,e \) are
  \( \bC \)-linearly dependent.
  Then the \( a\wedge Ja \)-term in
  \( \sum_{j=1}^2 (2\sigma_j^U\wedge\sigma_j^W + (\sigma_j^m)^2)=0 \)
  gives us \( p c\wedge Jc=0 \), i.e.\ \( c=0 \).
  Altogether, we see that the common kernel of \( \sigma_1,\sigma_2 \)
  in \( W \) is contained in the kernel of \( y_1,Jy_1,e,Je \) and
  this kernel has codimension at most two since \( y_1,e \) are
  \( \bC \)-linearly dependent.

  Finally, let \( y_2=0 \).
  Then \( x_2=-\frac{r}{q} x_1 \) and the equation
  \( \sigma_1^m\wedge \sigma_1^W+\sigma_2^m\wedge \sigma_2^W=0 \)
  gives us \( x_1\wedge (c\wedge Jc-\frac{r}{q} e\wedge Je)=0 \).
  So either \( x_1=0 \) or \( c,e \) are \( \bC \)-linearly dependent.
  with \( c\wedge Jc=\frac{r}{q} e\wedge Je \).

  In the first case \( x_2=0 \) as well, i.e.\ we have
  \( \sigma_1^m=0\sigma_2^m=0 \), and we get again from
  \( \sum_{j=1}^2 2\sigma_j^U\wedge\sigma_j^W = \sum_{j=1}^2
  (2\sigma_j^U\wedge\sigma_j^W + (\sigma_j^m)^2)=0 \) that \( c=0 \)
  and the statement follows.
  Otherwise, \( c,e \) are \( \bC \)-linearly dependent with
  \( c\wedge Jc=\frac{r}{q} e\wedge Je \) and the equation
  \( \sum_{j=1}^2 (2\sigma_j^U\wedge\sigma_j^W + (\sigma_j^m)^2)=0 \)
  is equivalent to
  \( p c\wedge Jc - 2\bigl(1+\frac{r^2}{q^2}\bigr) x_1\wedge Jx_1=0
  \).
  This shows that the span of \( c,Jc,e,Je,x_1,Jx_1 \) is at most
  two-dimensional and the statement follows.
\end{proof}
The last lemma implies that the \( J \)-invariant kernel of
\( \tau_1,\tau_2\in \Lambda^{1,1} U^* \),
\( \theta\in \Lambda^{2,0} U^* \) as above, i.e.\ satisfying
\begin{equation*}
  \tau_1\wedge \tau_1+\tau_2\wedge \tau_2=\theta\wedge \overline{\theta},
\end{equation*}
is of codimension at most \( 6 \) in \( U \) and this kernel
corresponds to an Abelian SKT Lie algebra.  Thus, we have obtained:

\begin{theorem}
  \label{th:2dJinv}
  Let \( (\mfg,g,J) \) be an \( 2n \)-dimensional almost Hermitian Lie
  algebra such that \( \derg \) is two-dimensional and
  \( J \)-invariant.  Then \( (\mfg,g,J) \) is SKT if and only if
  \begin{enumerate}[(i)]
  \item\label{item:2dJ-c} there exists
    \( \alpha\in U^*\setminus \{0\} \) such that the only non-zero Lie
    brackets (up to anti-symmetry) are given by
    \begin{equation*} [Y,X]=\alpha(Y) JX
    \end{equation*}
    for \( X\in \derg \), \( Y\in \derg^{\perp} \) or
  \item\label{item:2dJ-t} there exists an orthonormal basis
    \( (X,JX) \) of \( \derg \),
    \( \tau_1,\tau_2\in \Lambda^{1,1} \left(\derg^{\perp}\right)^* \)
    and \( \theta\in \Lambda^{2,0} \left(\derg^{\perp}\right)^* \)
    with \( \tau_1\neq 0 \), \( \tau_2\neq 0 \) or \( \theta\neq 0 \)
    and
    \begin{equation*}
      \tau_1\wedge \tau_1+\tau_2\wedge \tau_2=\theta\wedge \overline{\theta}.
    \end{equation*}
    such that the only non-zero Lie brackets (up to anti-symmetry) are
    given by
    \begin{equation*}
      [Y,Z]=\bigl(\tau_1(Y,Z)+\Re(\theta)(Y,Z)\bigr)X
      +\bigl(\tau_2(Y,Z)-\Im(\theta)(Y,Z)\bigr)JX
    \end{equation*}
    for \( Y,Z\in \derg^{\perp} \).
  \end{enumerate}
  Moreover, in case~\ref{item:2dJ-c}, we have
  \( \mfg\cong \mfr'_{3,0}\oplus \bR^{2n-3} \), whereas in
  case~\ref{item:2dJ-t} \( \mfg\cong \mfh\oplus \bR^{2n-2k} \) for a
  \( 2k \)-dimensional two-step nilpotent Lie algebra \( \mfh \) with
  \( k\in \{3,4\} \) and the splitting respects the \( SKT \)
  structure.
\end{theorem}

\begin{remark}
  So to describe explicitly all examples of SKT Lie algebras with
  two-dimensional \( J \)-invariant commutator ideal, we need to know
  all six- and eight-dimensional two-step nilpotent SKT Lie algebras
  with two-dimensional $J$-invariant commutator ideal.
  Fortunately, classifications of these Lie algebras were given
  in~\cite{FPS} and~\cite{EFV}.
\end{remark}

\section{Six-dimensional two-step solvable SKT Lie algebras}
\label{sec:6d2stepsolvSKT}

In this section, we work towards classifying six-dimensional two-step
solvable SKT Lie algebras \( (\mfg,g,J) \).

For this, let us first summarize what our previous results imply for
this situation.
First of all, Corollary~\ref{co:6dtotallyreal} gives us a complete
list in the cases when \( \derg \) is totally real.
Next, Theorem~\ref{th:2dJinv} describes all the cases with \( \derg \)
two-dimensional and complex, and Theorem~\ref{th:codim2notJinv}, we
determines cases with \( \derg \) four-dimensional but non-complex.
Finally, Corollary~\ref{co:6dalmostAbelianSKTLAs} contains a
classification of the cases where \( \derg \) has dimension~\( 5 \).

So, in order to classify all six-dimensional two-step solvable SKT Lie
algebras, we still have to determine those with three-dimensional
non-totally real \( \derg \) and those with four-dimensional complex
\( \derg \).
Unfortunately, it seems to be too hard to solve the latter case
completely and we will just determine the relevant equations
in~\S\ref{subsec:4dJinvariant} below.
In contrast, the first case is treated completely in the next
subsection and we note already now the following classification, that
includes results from~\cite{FPS} for the nilpotent case and uses some
of their notation.

\begin{theorem}
  \label{th:6d2stepsolvSKT}
  Let \( (\mfg,g,J) \) be a six-dimensional two-step solvable SKT Lie
  algebra.
  Then either \( \mfg \) is isomorphic as a Lie algebra to one of the
  following Lie algebras
  \begin{gather*}
    \bR^6,\quad \aff_{\bR}\oplus \bR^4,\quad \aff_{\bR}\oplus
    \mfh_3\oplus \bR,\quad 2\aff_{\bR}\oplus \bR^2,\quad
    3\aff_{\bR},\quad \mfh_3\oplus \bR^3,\quad 2\mfh_3,\\
    \mathfrak{r}_{3,0}'\oplus \bR^3,\quad (0,0,0,0,12,14+23),\quad
    (0,0,0,0,13+42,14+23),
  \end{gather*}
  or an algebra in Corollary~\ref{co:6dalmostAbelianSKTLAs}, or
  \( (\mfg,g,J) \) is isomorphic to one of the six-dimensional SKT Lie
  algebras in Theorem~\ref{th:codim2notJinv} or
  Theorem~\ref{th:6dwith3dnottotallyrealcommutator}, or \( \mfg' \) is
  four-dimensional, \( J \)-invariant, and, for a unit vector
  \( X\in (\mfg')^{\perp} \), the endomorphisms
  \( A_1:=\ad(X)|_{\mfg'} \) and \( A_2:=\ad(JX)|_{\mfg'} \)
  satisfy~\eqref{eq:4dJinvariant}.
\end{theorem}

\begin{remark}
  Note that Theorem~\ref{th:6d2stepsolvSKT} implies the
  four-dimensional two-step Lie algebras admitting an SKT structure
  are, in the notation of~\cite{MaSw},
  \begin{equation*}
    \bR^4,\quad \aff_{\bR}\oplus \bR^2,\quad 2\aff_{\bR},\quad
    \mfh_3\oplus \bR,\quad \mfr_{3,0}'\oplus \bR,\quad
    \mfr_{4,\lambda,0}',\quad \mfr_{4,-1/2,-1/2},\quad \mfr_{4,2\lambda,-\lambda}'.
  \end{equation*}
  agreeing with the results of that paper.
\end{remark}

\subsection{Three-dimensional commutator ideal not totally real}

Here, we consider two-step SKT shear data \( (\mfa,\omega) \) with
\( \dim(\mfa)=3 \) and \( \dim(\mfa_r)=1 \).
Then \( \mfa_J \) and \( U_J \) are real two-dimensional, and
\( U_r \) is real one-dimensional.
Take orthonormal bases \( (Y,JY) \) of~\( \mfa_J \), \( X
\)~of~\( \mfa_r \) and \( (Z,JZ \)) of~\( U_J \).
Then \( JX \) is a basis of \( U_r \) and \( (Y,JY,X) \) is an
orthonormal basis of~\( \mfa \).
Setting \( A:=-\omega_0(JX,\any)\in \End(\mfa) \),
\( B_1:=-\omega_0(Z,\any)\in \End(\mfa) \) and
\( B_2:=\omega_0(JZ,\any)\in \End(\mfa) \),
Lemma~\ref{le:Jintegrable}\ref{item:G-zero} and \ref{item:J-KX}, give
us with respect to the splitting \( \mfa=\mfa_J\oplus \mfa_r \)
\begin{equation*}
  A=\begin{pmatrix} z & u \\ 0 & a \end{pmatrix}
\end{equation*}
for some \( a\in\bR \), \( z,u\in \bC \).
Replacing \( X \) by \( -X \) if necessary we may take \( a \geq 0 \).
Equation~\eqref{eq:F2ndordergsymmetric} of
Lemma~\ref{le:conditionsonFX}, reduces to~\eqref{eq:z-a}, and we
conclude that \( \Re(z)\in \{0,-a/2\} \).

We now consider individually the cases \( A=0 \) and \( A\neq 0 \).

\begin{lemma}
  \label{le:A=0}
  With the above notation, if \( A=0 \), then the basis \( (Z,JZ) \)
  of \( U_{J} \) may be rotated so that
  \( \omega(X,\any)=\omega(JX,\any)=0 \), \( \omega(Z,Y)=0 \),
  \( \omega(JZ,Y)=ib Y \), \( b \in \bR_{>0} \), and
  \( \omega(Z,JZ)=q Y+ h X \), \( h\in \bR\setminus \{0\} \) and
  \( q\in \bC \).

  Conversely, pre-shear data of this form is two-step SKT shear data
  with \( \im(\omega)=\mfa \) and the obtained Lie algebra is almost
  Abelian and isomorphic to \( \mfg^{0}_{5,14}\oplus \bR \).
\end{lemma}

\begin{proof}
  Lemma~\ref{le:Jintegrable}\ref{item:o1-rrr} and \( a = 0 \) give
  \( \omega(X,JX)=0 \) and so
  \begin{equation}
    \label{eq:A0-nu1}
    \begin{split}
      \MoveEqLeft 6\,\nu_1(\omega)(X,JX,Z,JZ)\\
      &=-\norm{\omega(JX,JZ)}^2-\norm{\omega(JX,Z)}^2
        -\norm{\omega(X,JZ)}^2-\norm{\omega(X,Z)}^2.
    \end{split}
  \end{equation}
  Moreover, in the current situation \( \mfa\perp U \), so we get
  \begin{equation*}
    12\,\nu_2(\omega)(X,JX,Z,JZ)
    = g(\omega(J\omega(JX,Z),Z),X)+g(\omega(J\omega(JX,JZ),JZ),X).
  \end{equation*}
  Now \( \omega^r(J\omega(JX,JZ),JZ)=\omega^r(\omega(JX,JZ),Z) \)
  by~\eqref{eq:AbelianSKTshear1} since the other terms in this
  equation land in~\( U_r \).
  Moreover, \( \omega^r(\omega(JX,JZ),Z)=\omega^r(\omega(JX,Z),JZ) \)
  by~\eqref{eq:Abeliansheardata} as \( \omega(JX,\any)|_{\mfa}=0 \).
  Then, \( \omega^r(\omega(JX,Z),JZ)=-\omega^r(J\omega(JX,Z),Z) \),
  again by~\eqref{eq:AbelianSKTshear1}.
  But so
  \( g(\omega(J\omega(JX,JZ),JZ),X)=-g(\omega(J\omega(JX,Z),Z),X) \),
  which is equivalent to \( \nu_2(\omega)(X,JX,Z,JZ)=0 \).
  Thus, \eqref{eq:AbelianSKTshear2} implies \eqref{eq:A0-nu1} is zero,
  so
  \begin{equation*}
    \omega(X,Z)=\omega(JX,Z)=\omega(X,JZ)=\omega(JX,JZ)=0.
  \end{equation*}
  Hence, \( \omega(X,\any)=\omega(JX,\any)=0 \).

  Next,
  \( B_i=\begin{psmallmatrix} C_i & 0\\ y_i^T & 0\end{psmallmatrix} \)
  with \( C_i\in \bR^{2\times 2} \) and \( y_i\in \bR^2 \),
  \( i=1,2 \).
  By~\eqref{eq:Abeliansheardata}, \( [B_1,B_2]=0 \), so
  \( [C_1,C_2]=0 \).
  Moreover, \eqref{eq:AbelianSKTshear1} gives us
  \( [C_2,J]=J[C_1,J] \), Hence Lemma~\ref{lem:2d-comm-Jrel} implies
  that \( C_1 \) and \( C_2 \) commute with~\( J \).
  So write \( c_i\in \bC \) instead of \( C_i\in \bR^{2\times 2} \)
  and consider also \( y_i \) as complex numbers.

  Then we observe that
  \begin{equation*}
    \begin{split}
      \MoveEqLeft
      6\,\nu_1(\omega)(Y,iY,Z,JZ) \\
      &=-\norm{\omega(iY,JZ)}^2-\norm{\omega(Y,JZ)}^2
        -\norm{\omega(iY,Z)}^2-\norm{\omega(Y,Z)}^2\\
      &=-2\abs{c_1}^2-2\abs{c_2}^2-\abs{y_1}^2-\abs{y_2}^2
    \end{split}
  \end{equation*}
  and, using \( \omega(J\omega(W_1,W_2),W_3)=0 \) for all
  \( W_1,W_2\in \bR^{2n} \), \( W_3\in \mfa_J \), we get
  \begin{equation*}
    \begin{split}
      \MoveEqLeft
      12\,\nu_2(\omega)(Y,iY,Z,JZ) \\
      &=-g(\omega(J\omega(Y,Z),Z),iY)-
        g(\omega(J\omega(Y,JZ),JZ),iY) \eqbreak
        +g(\omega(J\omega(iY,JZ),JZ),Y)+g(\omega(J\omega(iY,Z),Z),Y)\\
      &=-g(C_1JC_1Y+C_2JC_2 Y,iY)+g(C_2JC_2 iY+C_1 J C_1 iY,Y)\\
      &=-2g((c_1^2+c_2^2)Y,Y)=-2\Re(c_1^2)-2 \Re(c_2^2).
    \end{split}
  \end{equation*}
  Since \( \Re(c_i^2)+\abs{c_i}^2=2\Re(c_i)^2 \) for \( i=1,2 \),
  equation~\eqref{eq:AbelianSKTshear2} gives
  \begin{equation*}
    0=6(\nu_1(\omega)+2\nu_2(\omega))(Y,iY,Z,JZ)
    =-4\Re(c_1)^2-4\Re(c_2)^2-\abs{y_1}^2-\abs{y_2}^2.
  \end{equation*}
  Thus \( \Re(c_1)=\Re(c_2)=y_1=y_2=0 \).
  One may now check that \( \omega \) satisfies all of the
  equations~\eqref{eq:Abeliansheardata}, \eqref{eq:AbelianSKTshear1}
  and~\eqref{eq:AbelianSKTshear2} and that \( \im(\omega)=\mfa \)
  precisely when \( (c_1,c_2)\neq (0,0) \) and \( b_2\neq 0 \).
  Rotating \( (Z,JZ) \) in such a way that \( \omega(Z,Y)=0 \), we may
  assume that \( c_1=0 \) and \( c_2= i b \) for some
  \( b\in \bR_{>0} \).
  Writing \( \omega(Z,JZ)=q Y+ h X \) with \( q\in \bC \) and
  \( h\in \bR \), the condition \( \im\omega = \mfa \) gives
  \( h\ne0 \) and we have the statement.
\end{proof}

Next, we consider the case \( A\neq 0 \).  For this, let us write
\begin{equation}
  \label{eq:componentsomegaJXZJZ}
  \begin{gathered}
    \omega(JX,Z)=w_1 Y+ f_1 X,\qquad \omega(JX,JZ)=w_2 Y+ f_2 X, \\
    \omega(Z,JZ)= q Y+ h X,
  \end{gathered}
\end{equation}
for some \( w_1,w_2,q\in \bC \), \( f_1,f_2,h\in \bR \).

\begin{lemma}
  \label{le:formofBj}
  With the above notation, if \( A =
  \begin{psmallmatrix}
    z&u\\0&a
  \end{psmallmatrix}
  \neq 0 \), then up to rotations in the \( \spa{Y,iY} \)- and
  \( \spa{Z,JZ} \)-planes, we have \( u\in \bR_{\geqslant0} \) and
  either
  \begin{enumerate}[(i)]
  \item \( \Re(z)=-a/2\neq 0 \),
    \begin{equation*}
      B_1=\begin{pmatrix}  i c_1 & v_1 \\ 0 & 0 \end{pmatrix},\qquad
      B_2=\begin{pmatrix} -\frac{b_2}{2}+ i c_2 & v_2 \\ 0 & b_2 \end{pmatrix}
    \end{equation*}
    \begin{equation*}
      v_1=-\frac{i c_1}{\frac{3}{2}a-ic} u,\qquad
      v_2=\frac{\frac{3}{2}b_2-ic_2}{\frac{3}{2}a-ic} u,
    \end{equation*}
    \( z=-\tfrac{a}{2}+ic \), \( f_1=-b_2 \), \( f_2=0 \),
    \( h=b_2^2/a \) and \( w_{1}+iw_2=i(v_1+iv_2) \), or
  \item \( \Re(z)=0 \), \( (a,c = \Im(z)) \neq (0,0) \),
    \begin{equation*}
      B_1=\begin{pmatrix}  i c_1 & v_1 \\ 0 & b_1 \end{pmatrix},\qquad
      B_2=\begin{pmatrix}  i c_2 & v_2 \\ 0 & b_2 \end{pmatrix}
    \end{equation*}
    \begin{equation*}
      v_1=\frac{b_1-ic_1}{a-ic} u,\qquad v_2=\frac{b_2-ic_2}{a-ic} u,
    \end{equation*}
    \( z=ic \), \( f_1=-b_2 \), \( f_2=b_1 \) and
    \( w_{1}+iw_2=i(v_1+iv_2) \),
  \end{enumerate}
  for some \( b_1,b_2,c_1,c_2,c\in \bR \).

  Conversely, pre-shear data \( (\mfa,\omega) \) with
  \( \dim(\mfa)=3 \) and \( \mfa \) not totally real with all these
  properties satisfies \eqref{eq:AbelianSKTshear1} everywhere,
  \eqref{eq:Abeliansheardata} on \( \Lambda^2 U\wedge \mfa \)
  and~\eqref{eq:AbelianSKTshear2} on
  \( \Lambda^2 \mfa_J \wedge \Lambda^2 \bR^6 \).
\end{lemma}

\begin{proof}
  First we rotate the \( \spa{Y,iY} \)-plane to get
  \( u\in \bR_{\geqslant 0} \).

  Note that \( \Re(z) \in \{0,-a/2\} \), implies either
  \( z - a \ne 0 \) or \( z = 0 = a \).
  Consider \( j\in \{1,2\} \).
  Since \( A \) and \( B_j \) commute, \( B_j \)~preserves the image
  and kernel of~\( A - aI_{3} \).
  As \( \spa{e_1,e_2} \) is one of these spaces, the image when
  \( z - a \ne 0 \), the kernel when \( z = 0 = a \), this space is
  preserved by~\( B_{j} \).  Thus,
  \begin{equation*}
    B_j=\begin{pmatrix} C_j & v_j \\ 0 & b_j \end{pmatrix}
  \end{equation*}
  with \( C_j\in \bR^{2\times 2} \), \( v_j\in \bR^2\cong \bC \) and
  \( b_j\in \bR \).
  Now \( [B_1,B_2]=0 \) gives \( [C_1,C_2]=0 \), furthermore
  Lemma~\ref{le:Jintegrable}\ref{item:o0-JJr} means we can apply
  Lemma~\ref{lem:2d-comm-Jrel} to get that \( C_j \) are
  \( J \)-invariant and may be identified with complex numbers.
  We may rotate~\( (Z,JZ) \) so that \( \Re(C_1)=0 \), i.e.\
  \( C_1=i c_1 \) for some \( c_1\in \bR \).

  Similarly to the proof of Lemma~\ref{le:A=0}, we get
  \begin{equation*}
    6\nu_1(\omega)(Y,iY,Z,JZ)=-2 \abs{C_1}^2-2 \abs{C_2}^2
  \end{equation*}
  and
  \begin{equation*}
    \begin{split}
      12\nu_2(\omega)(Y,iY,Z,JZ)
      &=-2 \Re(C_1^2)-2 \Re(C_2^2) \eqbreak
        +g(\omega(J(\omega(Z,JZ)),iY),iY)+g(\omega(J(\omega(Z,JZ)),Y),Y)\\
      &=-2 (\Re(C_1^2)+\Re(C_2^2)- h \Re(z)).
    \end{split}
  \end{equation*}
  Thus, \eqref{eq:AbelianSKTshear2} gives us
  \begin{equation*}
    \begin{split}
      0&=3\nu_1(\omega)(Y,iY,Z,JZ)+6\nu_2(\omega)(Y,iY,Z,JZ) \\
       &=-2\Re(C_1)^2-2\Re(C_2)^2- h \Re(z)
         =-2\Re(C_2)^2-h \Re(z),
    \end{split}
  \end{equation*}
  So \( \Re(C_2)=0 \) if \( \Re(z)=0 \), and \( \Re(C_2)^2=ah/4 \) if
  \( \Re(z)=-a/2 \).

  Next,
  \begin{equation*}
    \begin{split}
      6\nu_1(\omega)(Y,iY,X,Z)
      &=-g(\omega(iY,JX),\omega(iY,Z))-g(\omega(Y,JX),\omega(Y,Z))\\
      &=-2\Re(z\overline{C}_1)
    \end{split}
  \end{equation*}
  and
  \begin{equation*}
    \begin{split}
      12\nu_2(\omega)(Y,iY,X,Z)
      &=-g(\omega(J\omega(Y,Z),JX),iY)+g(\omega(J\omega(iY,Z),JX),Y)\\
      &+g(\omega(J\omega(X,Z),iY),iY)+g(\omega(J\omega(X,Z),Y),Y)\\
      &=-2\Re(zC_1)-2b_1 \Re(z),
    \end{split}
  \end{equation*}
  so that \eqref{eq:AbelianSKTshear2} gives us
  \begin{equation*}
    \begin{split}
      0&=6\nu(\omega)(Y,iY,X,Z)=-2\Re(z\overline{C}_1)-2\Re(zC_1)-2b_1 \Re(z)\\
       &=-2\Re(z)\left(2\Re(C_1)+b_1\right)=-2\Re(z) b_1.
    \end{split}
  \end{equation*}
  Hence, if \( \Re(z)=-a/2\neq 0 \), then \( b_1=0 \).
  Similarly, we obtain
  \begin{equation*}
    0=6\nu(\omega)(Y,iY,X,JZ)=-2\Re(z)\left(2\Re(C_2)+b_2\right),
  \end{equation*}
  and so \( \Re(C_2)=-b_2/2 \) if \( \Re(z)=-a/2\neq 0 \), which
  implies \( b_2^2/4=\Re(C_2)^2=ah/4 \) and hence \( h=b_2^2/a \).

  Note that
  \(
  \nu_1(\omega)(Y,iY,JX,Z)=\nu_1(\omega)(iY,Y,X,JZ)=-\nu_1(\omega)(Y,iY,X,JZ)
  \) and that
  \begin{equation*}
    \begin{split}
      \nu_2(\omega)(Y,iY,JX,Z)
      &=g(\omega(J\omega(Y,JX),JZ),iY)-g(\omega(J\omega(iY,JX),JZ),Y)
        \eqbreak
        +g(\omega(J\omega(JX,Z),iY),iY)+g(\omega(J\omega(JX,Z),Y),Y)\\
      &=g(i\omega(\omega(Y,JX),JZ),iY)-g(i\omega(\omega(iY,JX),JZ),Y) \eqbreak
        +g(\omega(J\omega^r(JX,Z),iY),iY)+g(\omega(J\omega^r(JX,Z),Y),Y)\\
      &=g(\omega(J\omega(Y,JZ),JX),iY)-g(\omega(J\omega(iY,JZ),JX),Y) \eqbreak
        -g(\omega(J\omega(X,JZ),iY),iY)-g(\omega(J\omega(X,JZ),Y),Y)\\
      &=-\nu_2(\omega)(Y,iY,X,JZ),
    \end{split}
  \end{equation*}
  so that \( \nu(\omega)(Y,iY,JX,Z)=-\nu(\omega)(Y,iY,X,JZ)=0 \).
  Similarly, the identity
  \( \nu(\omega)(Y,iY,JX,JZ)=\nu(\omega)(Y,iY,X,Z)=0 \) follows.
  Thus, \eqref{eq:AbelianSKTshear2} is also valid on
  \( \Lambda^2 \mfa_J\wedge U_r\wedge U_J \), and so on
  \( \Lambda^2 \mfa_J\wedge \Lambda^2\bR^6 \).

  Next, we notice that~\eqref{eq:AbelianSKTshear1} is satisfied if and
  only if parts \ref{item:o1-Jrr} and~\ref{item:o01-J-val} of
  Lemma~\ref{le:Jintegrable} hold.
  Now \ref{item:o1-Jrr} is equivalent to \( f_1=-b_2 \) and
  \( f_2=b_1 \) and \ref{item:o01-J-val} is equivalent to
  \begin{equation*}
    w_{1}+iw_2=i(v_1+iv_2).
  \end{equation*}
  Moreover, by~\eqref{eq:Abeliansheardata}, we must have
  \( [A,B_i]=0 \), which is equivalent to \( zv_i+ub_i=C_i u+v_i a \),
  i.e.\ to \( (z-a) v_i=(C_i-b_i) u \), and so to
  \( v_i=\tfrac{C_i-b_i}{z-a} u \), \( i=1,2 \), if \( z-a\neq 0 \).

  If \( z-a=0 \), then \( z = 0 = a \) and \( u \ne 0 \).
  As \( C_j=ic_j \), we have \( 0=(z-a)v_j=(ic_j-b_j)u \), giving
  \( b_j=c_j=0 \) for \( j=1,2 \).
  This implies \( \omega(X,Z) \), \( \omega(JX,Z) \),
  \( \omega(X,JZ) \), \( \omega(JX,JZ) \) all lie in \( \mfa_J \) and
  \( \mfa_J\subseteq \ker(\omega(JX,\any))\cap
  \ker(\omega(Z,\any))\cap \ker(\omega(JZ,\any)) \).
  Using these properties, \eqref{eq:Abeliansheardata} implies
  \begin{equation*}
    \begin{split}
      0&=\omega(JX,\omega(Z,JZ)) + \omega(Z,\omega(JZ,JX))
         + \omega(JZ,\omega(JX,Z)) \\
       &=\omega(JX,q Y+ h X)
         = uh Y.
    \end{split}
  \end{equation*}
  Since \( u\neq 0 \), we get \( h=0 \) and thus
  \( \im(\omega)\subseteq \mfa_J \), a contradiction.
  We conclude that \( z-a\neq 0 \), which is clear if
  \( \Re(z)=-a/2\neq 0 \) but implies \( a\neq 0 \) or
  \( \Im(z)\neq 0 \) in the case \( \Re(z)=0 \).
\end{proof}

Summarizing our results to this point, we observe that we still have
to impose~\eqref{eq:Abeliansheardata} on
\( \Lambda^3 U=U_r\wedge \Lambda^2 U_J \) and
\eqref{eq:AbelianSKTshear2} on
\( \mfa_J\wedge \Lambda^3 (\mfa_r\oplus U_r\oplus U_J) \) and on
\( \Lambda^4 (\mfa_r\oplus U_r\oplus U_J)=\mfa_r\wedge U_r\wedge
\Lambda^2 U_J \).

First of all, \eqref{eq:Abeliansheardata} evaluated on
\( U_r\wedge \Lambda^2 U_J \) yields
\begin{equation*}
  \begin{split}
    0&=\omega(JX,\omega(Z,JZ)) + \omega(Z,\omega(JZ,JX)) +
       \omega(JZ,\omega(JX,Z))\\
     &=A\cdot (q,h)^T-B_1 (w_2,b_1)^T+B_2 (w_1,-b_2),
  \end{split}
\end{equation*}
where we set \( b_1=0 \) if \( \Re(z)=-{a}/{2} \).
This equation is equivalent to
\begin{equation}
  \label{eq:lasteqnsforJacobiidentity}
  0=zq+ hu- ic_1 w_2-b_1 v_1+ C_2 w_1-b_2 v_2
  \eqand
  ah=\norm{b}^2
\end{equation}
Note that for \( \Re(z)=-{a}/{2}\neq 0 \), the second equation in
automatically satisfied as \( b_1=0 \) and \( h={b_2^2}/{a} \).

Let us now first look at \eqref{eq:AbelianSKTshear2} on
\( \mfa_J\wedge \mfa_r\wedge U_r\wedge U_J \).
For this, let \( \tilde{Y}\in \mfa_J \).  Then
\begin{equation*}
  \begin{split}
    6\nu_1(\omega)(\widetilde{Y},X,JX,Z)
    &= g(\omega(i\widetilde{Y},JX),\omega(JX,Z))
      + g(\omega(i\widetilde{Y},JZ),\omega(X,JX)) \eqbreak
      - g(\omega(JX,X),\omega(\widetilde{Y},Z))
      - g(\omega(JX,JZ),\omega(\widetilde{Y},JX))\\
    &=g(i z \tilde{Y},w_1 Y) + g(i C_2 \widetilde{Y},u Y) + g(u Y,i
      c_1 \tilde{Y}) - g(w_2 Y,z \tilde{Y})
  \end{split}
\end{equation*}
and
\begin{equation*}
  \begin{split}
    12\nu_2(\omega)(\widetilde{Y},X,JX,Z)
    &= - g(\omega(J\omega(X,JX),JZ),\widetilde{Y})
      - g(\omega(J\omega(JX,Z),JX),\widetilde{Y}) \eqbreak
      + g(\omega(J\omega(Z,X),X),\widetilde{Y})\\
    &= - g(iu C_2 Y,\tilde{Y}) - g(a w_2 Y,\tilde{Y}) - g(iw_1 z
      Y,\tilde{Y}) + g(b_1 u Y,\tilde{Y}).
  \end{split}
\end{equation*}
Thus, \( \nu(\omega)(\widetilde{Y},X,JX,Z)=0 \) for all
\( \widetilde{Y}\in \mfa_J \) is equivalent to
\begin{equation}
  \label{eq:SKTcond1}
  0 = - 2 i \Re(z) w_1 - 2i \Re(C_2) u - i c_1 u - w_2 \overline{z}
  - a w_2 + b_1 u.
\end{equation}
Similarly, \( \nu(\omega)(\widetilde{Y},X,JX,JZ)=0 \) for all
\( \widetilde{Y}\in \mfa_J \) is seen to be equivalent to
\begin{equation}
  \label{eq:SKTcond2}
  \begin{split}
    0&=-i \overline{z} w_2 + c_1 u + \overline{C}_2 u +
       \overline{z}w_1 - u c_1 + a w_1 - i w_2 z + b_2 u\\
     &=-2 i \Re(z) w_2+\overline{C}_2 u+w_1 \overline{z}+a w_1+ b_2 u.
  \end{split}
\end{equation}
Analogous computations give \( \nu(\omega)(\widetilde{Y},X,Z,JZ)=0 \)
for all \( \widetilde{Y}\in \mfa_J \) if and only if
\begin{equation}
  \label{eq:SKTcond3}
  0=2 i v_2 \Re(C_2)-2 i q\Re(z)+b_1 w_1+b_2 w_2+\overline{C}_2 w_2-i c_1 w_1
\end{equation}
and that \( \nu(\omega)(\widetilde{Y},JX,Z,JZ)=0 \) for all
\( \widetilde{Y}\in \mfa_J \) if and only if
\begin{equation}
  \label{eq:SKTcond4}
  0=2 i w_2 \Re(C_2)-\overline{C}_2 v_2+i c_1 v_1+q\overline{z}-b_2 w_1-hu+b_1 w_2
\end{equation}
Finally, we have to consider the equation
\( \nu(\omega)(X,JX,Z,JZ)=0 \).  Using
\begin{equation*}
  \begin{split}
    \omega^r(J\omega(JX,JZ),JZ)
    &=\omega^r(\omega(JX,JZ),Z)\\
    &=\omega^r(\omega(Z,JZ),JX)+\omega^r(\omega(JX,Z),JZ)\\
    &=-\omega^r(J\omega(Z,JZ),X)-\omega^r(J\omega(JX,Z),Z),
  \end{split}
\end{equation*}
one easily sees that
\begin{equation*}
  \begin{split}
    \nu_2(\omega)(X,JX,Z,JZ)=0.
  \end{split}
\end{equation*}
Thus, \( 0=\nu(\omega)(X,JX,Z,JZ)=\nu_1(\omega)(X,JX,Z,JZ) \), and a
short computation yields that this equation is equivalent to
\begin{equation}
  \label{eq:SKTcond5}
  \begin{split}
    0&= 2g(\omega(X,JX),\omega(Z,JZ)) - \norm{\omega(X,Z)}^2
       - \norm{\omega(X,JZ)}^2 \eqbreak
       -\norm{\omega(JX,Z)}^2 -\norm{\omega(JX,JZ)}^2.
  \end{split}
\end{equation}
Let us now consider individually the cases (a)~\( \Re(z)=0 \) and
(b)~\( \Re(z)=-{a}/{2}\neq 0 \).

(a) \( \Re(z)=0 \).
Here, \( \Re(C_2)=0 \) and so \( C_2=ic_2 \) for some
\( c_2\in \bR \).
Then \eqref{eq:SKTcond1} and~\eqref{eq:SKTcond2} imply
\begin{equation}
  \label{eq:Rez=0}
  w_2=\frac{b_1-ic_1}{a-ic} u=v_1,\qquad w_1=-\frac{b_2-ic_2}{a-ic} u=-v_2,
\end{equation}
which is consistent with \( w_{i}+iw_2=i(v_1+iv_2) \).
Moreover, \eqref{eq:SKTcond3} is equivalent to
\( (b_2-ic_2)w_2=-(b_1-ic_1) w_1 \), and so also automatically
satisfied.
Furthermore, the first equation
in~\eqref{eq:lasteqnsforJacobiidentity} and~\eqref{eq:SKTcond4} are
both easily seen to be equivalent to
\begin{equation*}
  \begin{split}
    0&=ic q+hu-(b_1+ic_1) v_1-(b_2+ic_2)v_2
       =ic q+h u-\tfrac{b_1^2+b_2^2+c_1^2+c_2^2}{a-ic}u\\
     &=icq+\tfrac{ha-ihc-b_1^2-b_2^2-c_1^2-c_2^2}{a-ic}u
       =icq-\tfrac{ihc+c_1^2+c_2^2}{a-ic}u.
  \end{split}
\end{equation*}
So if \( c\neq 0 \), we obtain
\begin{equation*}
  q=\tfrac{ich+c_1^2+c_2^2}{c^2+ac i}u.
\end{equation*}
If \( c=0 \), we would get \( c_1^2+c_2^2=0 \), hence \( c_1=c_2=0 \).
However, this case cannot occur as then the only non-zero Lie brackets
on the shear would be (up to anti-symmetry) \( [Z,JZ] \) and
\begin{equation*}
  [X,JX]=\tfrac{b_1}{a} [X,Z]=\tfrac{b_1}{a} [JX,JZ]
  =\tfrac{b_2}{a} [X,JZ]=-\tfrac{b_2}{a} [JX,Z]
\end{equation*}
giving \( \dim(\derg)\leq 2<3 \), a contradiction.

Thus, \( c\neq 0 \) and we still have to satisfy~\eqref{eq:SKTcond5}.
Due to \( \omega(JX,JZ)=\omega(X,Z) \) and
\( \omega(JX,Z)=-\omega(X,JZ) \) and \( ha=\norm{b}^2 \), this
equation is
\begin{equation*}
  0=2 \Re(q) u+2 h a- 2 \abs{v_1}^2- 2 b_1^2-2 \abs{v_2}^2- 2 b_2^2
  = 2(\Re(q) u-\abs{v_1}^2-\abs{v_2}^2),
\end{equation*}
so
\begin{equation*}
  \Re(q) u=\abs{v_1}^2+\abs{v_2}^2=\tfrac{b_1^2+b_2^2+c_1^2+c_2^2}{a^2+c^2} u^2.
\end{equation*}
On the other hand, \( q=\tfrac{ich+c_1^2+c_2^2}{c^2+ac i}u \) implies
\begin{equation*}
  \begin{split}
    \Re(q) u=\Re\left(\tfrac{ich+c_1^2+c_2^2}{c^2+ac i}u\right)u=\tfrac{c_1^2+c_2^2+ah}{a^2+c^2}u^2=\tfrac{c_1^2+c_2^2+b_1^2+b_2^2}{a^2+c^2}u^2,
  \end{split}
\end{equation*}
and so \eqref{eq:SKTcond5} is automatically satisfied.

(b)~\( \Re(z)=-{a}/{2}\neq 0 \).
In this case, \eqref{eq:SKTcond1} and~\eqref{eq:SKTcond2} are
explicitly given by
\begin{equation*}
  i a w_1  - \bigl(\tfrac{a}{2}-ic\bigr) w_2+ i(b_2-c_1) u =0,\qquad
  i a w_2 + \bigl(\tfrac{a}{2}-ic\bigr) w_1
  + \bigl(\tfrac{b_2}{2}-ic_2\bigr) u=0.
\end{equation*}
The solution is given by
\begin{equation*}
  \begin{split}
    w_1&=-\frac{\bigl(\tfrac{a}{2}-ic\bigr)\bigl(\tfrac{b_2}{2}-ic_2\bigr)
         -a(b_2-c_1)}{\bigl(\tfrac{a}{2}-ic\bigr)^2-a^2}u,\\
    w_2&=i\frac{\bigl(\tfrac{a}{2}-ic\bigr)(b_2-c_1)
         -a \bigl(\tfrac{b_2}{2}-ic_2\bigr)}{
         \bigl(\tfrac{a}{2}-ic\bigr)^2-a^2}u,
  \end{split}
\end{equation*}
and \( w_{1}+iw_2=i(v_1+iv_2) \) holds.

Next, using \( w_2=v_1+iv_2+iw_1 \), \eqref{eq:SKTcond4} is explicitly
given by
\begin{equation*}
  \begin{split}
    0&=-i b_2 w_2+\bigl(\tfrac{b_2}{2}+ic_2\bigr) v_2+ic_1 v_1
       -\bigl(\tfrac{a}{2}+ic\bigr)q-b_2 w_1-hu\\
     &=-i b_2 v_1+b_2 v_2+\bigl(\tfrac{b_2}{2}+ic_2\bigr) v_2
       +ic_1 v_1-\bigl(\tfrac{a}{2}+ic\bigr)q-hu\\
     &=\bigl(\tfrac{3}{2}b_2+ic_2\bigr) v_2-i (b_2-c_1)v_1
       -\bigl(\tfrac{a}{2}+ic\bigr)q-hu,
  \end{split}
\end{equation*}
and so we do get
\begin{equation*}
  \begin{split}
    q&=\frac{-i(b_2-c_1)v_1 + \bigl(\tfrac{3}{2}b_2+ic_2\bigr)v_2 -
       hu}{\tfrac{a}{2}+ic}
       =\frac{(c_1-b_2)c_1 + \tfrac{9}{4}b_2^2 + c_2^2 -
       h\bigl(\tfrac{3}{2}a-ic\bigr)}{\bigl(\tfrac{a}{2}+ic\bigr)
       \bigl(\tfrac{3}{2}a-ic\bigr)}u.
  \end{split}
\end{equation*}
A lengthy but straightforward computation shows that then the first
equation in~\eqref{eq:lasteqnsforJacobiidentity} and also
\eqref{eq:SKTcond3} and~\eqref{eq:SKTcond5} are satisfied.

\begin{theorem}
  \label{th:6dwith3dnottotallyrealcommutator}
  Let \( (\mfg,g,J) \) be a six-dimensional almost Hermitian Lie
  algebra.
  Then \( (\mfg,g,J) \) is a two-step solvable SKT Lie algebra with
  \( \derg \) three-dimensional and not totally real if and only if
  \( (\mfg,g,J) \) admits a unitary \( (Y,iY,X,JX,Z,JZ) \) basis such
  that \( Y \) spans \( \derg_J \) as a complex vector space and
  \( X \) spans the real vector space~\( \derg_r \) and the only
  non-zero Lie brackets (up to anti-symmetry and complex-linear
  extension to \( \derg_J \)) are given by either
  \begin{enumerate}[(i)]
  \item \( [JZ,Y]=ibY \), \( [Z,JZ]=qY+hX \) for
    \( b,h\in \bR\setminus \{0\} \), \( q\in \bC \),
  \item
    \begin{equation*}
      \begin{split}
        [JX,Y]&=ic Y,\qquad  [Z,Y]=ic_1 Y, \qquad [JZ,Y]=i c_2 Y,\\
        [JX,X]&= uY+a X,\qquad [Z,X]=[JZ,JX]=
                -\frac{ic_1}{\tfrac{3}{2}a-ic} u Y+ b_1 X, \\
        [JZ,X]&=-[Z,JX]= \frac{\tfrac{3}{2}b_2-ic_2}{\tfrac{3}{2}a-ic}
                u Y+ b_2 X,\qquad
                [Z,JZ]=-\tfrac{ich+c_1^2+c_2^2}{c^2+ac i}u Y-h X
      \end{split}
    \end{equation*}
    for certain \( c\in \bR\setminus \{0\} \),
    \( b_1,b_2,c_1,c_2,h\in \bR \), \( a,u \in \bR_{\geqslant 0} \)
    with \( ha=b_1^2+b_2^2 \) and at least one of
    \( a,b_{1},b_{2},h \) non-zero, or
  \item
    \begin{equation*}
      \begin{split}
        [JX,Y]&=\left(-\frac{a}{2}+ic\right) Y,\qquad  [Z,Y]=ic_1 Y,
                \qquad [JZ,Y]=\left(-\frac{b_2}{2}+i c_2\right) Y,\\
        [JX,X]&= uY+a X,\qquad [Z,X]=-\frac{i c_1}{\frac{3}{2}a-ic} u
                Y, \\
        [JZ,X]&=\frac{\frac{3}{2}b_2-ic_2}{\frac{3}{2}a-ic} u Y+ b_2
                X,\qquad [Z,JZ]=-q Y-h X,\\
        [Z,JX]&=-\frac{\left(\tfrac{a}{2}-ic\right)\left(\tfrac{b_2}{2}-ic_2\right)-a(b_2-c_1)}{\left(\tfrac{a}{2}-ic\right)^2-a^2}u
                Y- b_2 X,\\
        [JZ,JX]&=i\frac{\left(\tfrac{a}{2}-ic\right)(b_2-c_1)-a
                 \left(\tfrac{b_2}{2}-ic_2\right)
                 }{\left(\tfrac{a}{2}-ic\right)^2-a^2}u Y
      \end{split}
    \end{equation*}
    for \( a\in \bR\setminus \{0\} \), \( b_2,c,c_1,c_2\in \bR \),
    \( u \in \bR_{\geqslant 0} \), with
    \begin{equation*}
      h=\frac{b_2^2}{a},\qquad q=\frac{(c_1-b_2)c_1+\tfrac{9}{4}b_2^2+c_2^2-h\left(\tfrac{3}{2}a-ic\right)}{\left(\tfrac{a}{2}+ic\right)\left(\tfrac{3}{2}a-ic\right)}u.
    \end{equation*}
  \end{enumerate}
\end{theorem}

\subsection{Four-dimensional complex commutator ideal}
\label{subsec:4dJinvariant}

As announced at the beginning of~\S\ref{sec:6d2stepsolvSKT}, we only
write down the equations in this case.

First of all, note that \( \mfa=\mfa_J \), \( U=U_J \) and take an
orthonormal basis \( Y_1,Y_2=JY_{1} \) of \( U \).
Then \( \omega \) is completely determined by
\( A_i:=\omega_0(Y_i,\any)\in \End(\mfa) \), \( i=1,2 \), and by
\( X:=\omega_1(Y_1,Y_2)\in \mfa \).
Moreover, \eqref{eq:Abeliansheardata} reduces to \( [A_1,A_2]=0 \)
and~\eqref{eq:AbelianSKTshear1} reduces to \( [A_2,J]=J[A_1,J] \),
i.e.\ \( A_2^{J-}=J A_1^{J-} \), where the superscript \( J- \)
denotes the \( J \)-anti-invariant part of an endomorphism of
\( \mfa=\mfa_J \).  Finally, \eqref{eq:AbelianSKTshear2} reduces to
\begin{equation*}
  JA_1^T A_1+A_1^T A_1 J+J A_2^T A_2+A_2^T A_2 J+A_1 J A_1 +A_1^T J A_1^T+ A_2 J A_2 + A_2^T J A_2^T=0.
\end{equation*}
where the transpose is taken with respect to the Riemannian metric
\( g|_{\mfa} \) on \( \mfa \).
Note that this is equivalent to
\( J A_1^T A_1+ A_1 J A_1 + J A_2^T A_2+ A_2 J A_2 \) being symmetric.
Summarizing, \( A_1 \) and \( A_2 \) have to satisfy the following
equations:
\begin{equation}
  \label{eq:4dJinvariant}
  \begin{gathered}
    [A_1,A_2]=0,\quad A_2^{J-}=J A_1^{J-},\\
    \sum_{i=1}^2 JA_i^T A_i+A_i^T A_i J+A_i J A_i+A_i^T J A_i^T=0.
  \end{gathered}
\end{equation}
Writing \( A_{i} = (A_{i})_{+} + (A_{i})_{-} \), with
\( (A_{i})_{\varepsilon}^{T} = \varepsilon (A_{i})_{\varepsilon} \),
for \( \varepsilon \in \{\pm1\} \), this decomposition commutes with
the decomposition \( A_{i} = A_{i}^{J} + A_{i}^{J-} \) and the final
equation in~\eqref{eq:4dJinvariant} become the two equations
\begin{equation*}
  \begin{gathered}
    \sum_{i=1}^{2} [(A_{i})^{J}_{+}, (A_{i})^{J-}_{+}] -
    [(A_{i})^{J-}_{-},(A_{i})^{J}_{-}] = 0,\\
    \sum_{i=1}^{2} [(A_{i})^{J}_{+}, (A_{i})^{J-}_{-}] -
    [(A_{i})^{J-}_{-},(A_{i})^{J}_{+}] + 2 \bigl( (A_{i})^{J}_{+}
    \bigr)^{2} - 2 \bigl( (A_{i})^{J-}_{-} \bigr)^{2} = 0.
  \end{gathered}
\end{equation*}

\medbreak\noindent \textsc{Acknowledgements.} The first author was
partly supported by a \emph{Forschungs\-stipendium} (FR 3473/2-1) from
the Deutsche Forschungsgemeinschaft (DFG).

\appendix

\section{Notation for particular Lie algebras}

Table~\ref{table_LAs} contains a list of all indecomposable Lie
algebras \( \mfg \) of various dimensions which occur in this article
and which are not well-known, in contrast to the Abelian Lie algebra
\( \bR \), the two-dimensional Lie algebra \( \aff_{\bR} \) of affine
motions of the real line and the \( (2k+1) \)-dimensional Heisenberg
Lie algebra \( \mfh_{2k+1} \).
The names for Lie algebras in dimension \( 3 \) are taken
from~\cite{Bi}, in four dimensions from~\cite{ABDO}, in five
dimensions from~\cite{Mu1}, in six dimensions from~\cite{Mu2} and in
seven dimensions from~\cite{Gong}.

In all cases, the Lie bracket is encoded dually by the tuple of
differentials \( (de^1,\ldots,de^n) \) for a basis
\( e^1,\ldots,e^n \) of \( \mfg \) and we use Salamon's notation for
this tuple.
Moreover, in some cases, we do not give all the necessary conditions
on the parameters to ensure that two Lie algebras in the corresponding
parameter families are non-isomorphic.
\begin{table}[h]
  \centering
  \begin{tabular}{@{\vrule width 0pt height 2.5ex depth
    1.5ex}*{4}{l}@{}}
    \toprule
    \( \mfg \) & dim & differentials & conditions \\
    \midrule
    \( \mfr'_{3,\lambda} \) & \( 3 \) & \( (0,\lambda.21+31,-21+\lambda.31) \) & \( \lambda\geq 0 \) \\
    \( \mfr_{4,\mu,\lambda} \) & \( 4 \) & \( (0,21,\mu.31,\lambda.41) \) & \( 0<\abs{\lambda}\leq \abs{\mu}\leq 1 \) \\
    \( \mfr'_{4,\mu,\lambda} \) & \( 4 \) & \( (0,\mu.21,\lambda.31+41,-31+\lambda.41) \) & \( \mu>0 \) \\
    \( \mfg_{5,14}^{\alpha} \) & \( 5 \) & \( (0,0,21,\alpha.41+51,-41+\alpha.51) \) & \( \alpha \geq 0 \) \\
    \( \mfg_{5,17}^{\alpha,\beta,\gamma} \) & \( 5 \) & \( \begin{array}[t]{@{}l@{}}(0,\alpha.21+31,-21+\alpha.31,\\\qquad\beta.41+ \gamma.  51,-\gamma.41+\alpha.51)\end{array} \) & \( \alpha \geq 0 \), \( \gamma\neq 0 \) \\
    \( \mfg_{6,1}^{\alpha,\beta,\gamma,\delta} \) & \( 6 \) & \( (0,21,\alpha.31,\beta.41,\gamma.51,\delta.61) \) & \( 0<\abs{\delta}\leq \abs{\gamma}\leq \abs{\beta}\leq \abs{\alpha}\leq 1 \) \\
    \( \mfg_{6,8}^{\alpha,\beta,\gamma,\delta} \) & \( 6 \) & \( \begin{array}[t]{@{}l@{}}(0,\alpha.21,\beta.31,\gamma.41,\\\qquad\delta.51+61,-51+\delta.61)\end{array} \) & \( 0< \abs{\gamma}\leq \abs{\beta}\leq \abs{\alpha} \) \\
    \( \mfg_{6,11}^{\alpha,\beta,\gamma,\delta} \) & \( 6 \) & \( \begin{array}[t]{@{}l@{}}(0,\alpha.21,\beta.31+41,-31+\beta.41,\\\qquad\quad \gamma.51+\delta.61,-\delta.51+\gamma.61)\end{array} \) & \( \alpha \delta\neq 0 \) \\
    \( (37D) \) & 7 & \( (0,0,0,0,e^{12}+e^{34},e^{13},e^{24}) \) & --- \\
    \bottomrule
  \end{tabular}
  \medskip
  \caption{Notation for certain Lie algebras}
  \label{table_LAs}
\end{table}
\newpage
\providecommand{\bysame}{\leavevmode\hbox to3em{\hrulefill}\thinspace}
\providecommand{\MR}{\relax\ifhmode\unskip\space\fi MR }
\providecommand{\MRhref}[2]{%
  \href{http://www.ams.org/mathscinet-getitem?mr=#1}{#2}
}
\providecommand{\href}[2]{#2}

\end{document}